\documentclass{amsart}    
\allowdisplaybreaks
\usepackage[
backend=biber,
style=numeric,
sorting=nyt
]{biblatex}
\usepackage{microtype}
\usepackage[colorlinks, citecolor=blue]{hyperref}

\title{Spectral embedding through weak* limit of finite-dimensional approximations}

\author{Fabrice Nonez}
\address{Department of Mathematics and Statistics, Concordia University, Montreal, Canada}
\email{fabrice.nonez@mail.concordia.ca}


\newtheorem{thm}{Theorem}[section]    
\newtheorem{theorem}{Theorem}[section] 
           
\newtheorem{lemma}[thm]{Lemma}        
\newtheorem{prop}[thm]{Proposition}    
\newtheorem{cor}[thm]{Corollary}     
     
\theoremstyle{definition}
\newtheorem{defn}[thm]{Definition}    
\newtheorem{definition}[thm]{Definition}

\theoremstyle{remark}
\newtheorem{remark}[thm]{Remark}

\def \Proj{\operatorname{proj}}
\def \Span{\operatorname{span}}
\def \Dom{\operatorname{dom}}
\def \St{\operatorname{st}}
\def \Id{\operatorname{id}}
\def \Borel{\operatorname{Borel}}
\def \HOM{\hat{\Omega}}
\def \HMU{\hat{\mu}}
\def \HNU{\hat{\nu}}
\def \HU{\hat{U}}
\def \HT{\hat{T}}
\def \Hm{\hat{m}}
\def \LL{\langle}	
\def \RR{\rangle}
\newcounter{pointnumber}

\begin{document}
\begin{abstract}    
The scope of this text is to study a process that induces another proof of the Spectral Embedding Theorem: that any densely defined symmetric operator can be extended by a multiplication operator through an embedding of the Hilbert space into an $L_2$ space. Furthermore, that process is meant to be used for specific operators, where natural spectral embeddings or equivalences may be found.

That process has previously been considered in \cite{nonez2024spectralequivalencesnonstandardsamplings} and in \cite{goldbring2025nonstandardapproachdirectintegral}, where it has been introduced through nonstandard techniques. Our contribution aims to be the reformulation of the theory through classical analysis arguments, without the use of nonstandard techniques nor ultraproducts.
\end{abstract}
\maketitle
\section{Overview and motivations}

First, we state the central theorem of this paper, which we also call the Spectral embedding Theorem.
\begin{thm}[Spectral Theorem for symmetric operators]\label{theorem_spectral_embedding}
	Suppose that $H$ is a separable real or complex Hilbert space, and that $A$ is a densely defined symmetric operator on $H$. Then, there exists a compact metric space $\HOM,$ a probability measure $\HMU$ on $\Borel(\HOM)$, an isometry $\HU:H\rightarrow L_2(\HOM,\HMU)$ and a square-integrable function $\Hm:\HOM\rightarrow\mathbb{R}$ generating multiplication operator $\HT$ on $L_2(\HOM,\HMU)$ such that $\HU\circ A \subset \HT\circ \HU.$
\end{thm}
\begin{remark}
	By the Appendix of \cite{nonez2024spectralequivalencesnonstandardsamplings} (which only uses elementary techniques), if $A$ is self-adjoint, it is a corollary that:
	\begin{itemize}
		\item $\HU\circ A = \HT\circ \HU.$ 
		\item $\HU(H)$ reduces for $\HT$, meaning that $\Proj_{\HU(H)} \HT=\HT\Proj_{\HU(H)}$.
		\item for $V\in\operatorname{Borel}(\mathbb{R}),$ if $\hat{P}(V)$ is the multiplication operator on $L_2(\HOM,\HMU)$ induced by $\mathbf{1}_{\Hm^{-1}(V)}$ and if $P(V)=\HU^* \hat{P}(V) \HU,$ then $P:\operatorname{Borel}(\mathbb{R})\rightarrow\{\text{ projections on $H$ } \}$ is the spectral measure of $A$. This establishes the spectral measure version of the Spectral Theorem.
	\end{itemize}
\end{remark}

While the Spectral Theorem is not usually stated for merely symmetric operators instead of self-adjoint, it is already known that any symmetric operators has a self-adjoint extension if we allow the space to be larger, and as such this version is equivalent to the usual formulation for self-adjoint operators.

The aim of this paper lies the description of a "process". This process has a set of parameters that is the "$A$-converging sampling-scale-sequence" (which we define below). Then, it induces resulting objects $\HOM$, $\HMU$, $\HU$ and $\Hm$ that satisfy the Spectral embedding Theorem.

 This process was first introduced in \cite{nonez2024spectralequivalencesnonstandardsamplings}, and then it was adapted in \cite{goldbring2025nonstandardapproachdirectintegral} to show the direct integral version of the Spectral Theorem. Both of these papers describe this process through nonstandard analysis, and fundamentally rely on concepts such as internal sets and Loeb measures to establish its validity. The goal here is to create a description of it that achieves the same aim as in \cite{nonez2024spectralequivalencesnonstandardsamplings}, sharing the same underlying idea, while only using tools and ideas of classical functional analysis to do it. 
 
 Counterintuitively, we think that this showcases one of the strengths of nonstandard analysis: even if one aims for a classical method, it is often a good strategy to try nonstandard methods first, then try to extract classical methods from them.

There are two main motivations for this process. The first one is to provide an interesting proof of the spectral theorem, which is done by showing there are general parameters that will work for any symmetric densely-defined $A$. Since no complex analysis, resolvent theory or Cayley transform are involved, the proof is the same whether the Hilbert space is real or complex. Instead, the proof relies on the finite-dimensional version of the spectral theorem.

Aside from both papers mentioned above, this paper shares ideas with \cite{HIRVONEN_HYTTINEN_2024}, whereas finite dimensional approximations of bounded self-adjoint operators with cyclic vectors are used to construct a suitable measure from their ultraproduct, noting that any self-adjoint bounded operator can split into subspaces with cyclic vectors.

Another proof that uses nonstandard/ultraproduct finite dimensional approximations to prove the Spectral Theorem is found in \cite{GOLDBRING2021590}, which itself inspired the work done in \cite{nonez2024spectralequivalencesnonstandardsamplings}, while another, short proof, can be found in \cite{matsunaga2024shortnonstandardproofspectral}. One difference with all of these papers is that we aim to not use any nonstandard analysis nor ultraproduct machinery (both are conceptually equivalent).

As far as we know, there has not been many recent new proofs of the Spectral Theorem outside of ultraproducts/nonstandard analysis. One we could find is \cite{leinfelder2017secondlookageometric}, which aimed to revisit and simplify their work in \cite{leinfelder1979geometric}. One reason we think the proof presented in our paper is an interesting addition to the literature is that it directly results in the suitable multiplication operator, even in cases where the resulting isometry is not surjective, which is certain if $A$ has no self-adjoint extensions. This contrasts with the classical method of splitting the Hilbert space into subspaces with cyclic vectors.

The second motivation is to help study specific operators. This is done by tweaking the parameters so that not only they respect the criteria, but they are "well-suited" for the operator. When done right, the resulting objects can be explicit in their description, and allow precise study and calculations concerning the operator. To observe this phenomenon we consider the shift operator and the differential operator. Object of future research will be to consider other classes of operators. 

We note that the work done \cite{HIRVONEN_HYTTINEN_2024} and \cite{hyttinen2024approximationsfeynmanpathintegrals} also aims to apply ultraproducts methods to specific operators, notably ones that come from quantum mechanics. We think that the methods we present here (and in \cite{nonez2024spectralequivalencesnonstandardsamplings}) can help in this endeavor. 
\section{Key ideas and conventions}\label{section_idea}

We present here a (hopefully) intuitive summary of the process presented in this text.

First, we start with the separable Hilbert space $H$, and symmetric operator $A$ on $H$. We then consider a sequence, each of the form $(H_n, A_n, \Omega_n)$, where $H_n$ is some finite-dimensional subspace of $H$, $A_n$ is some symmetric application on $H_n$ with fixed orthonormal eigenbasis given by $\Omega_n.$ We want the the sequence to approximate $A$ in some sense. 

That sense is \textbf{not} that $A_n$ should converge to $A$  in any of the usual operator topologies, which is hard to achieve when $A$ is unbounded. Instead, we want $A$ to be "approximable" by the $A_n$: whenever $x\in\Dom(A),$ there exists some sequence $(x_n\in H_n)_{n\in\mathbb{N}}$ for which the sequence of pairs $(x_n, A_n x_n)$ converge to $(x,Ax)$.

The idea is that just as the measure of Theorem \ref{theorem_spectral_embedding} would be on the set of eigenvectors if $A$ was decomposable, we want to construct a sequence of measures $\mu_n$ on each $\Omega_n.$ We have to do this judiciously, which is why we introduce "scales" which we define in Section \ref{section_sampling}. The problem then becomes that of defining $\HOM$, which would be some "limit" of $\Omega_n$. 

In \cite{nonez2024spectralequivalencesnonstandardsamplings}, it was done intrinsically using a nonstandard hull. The natural idea in classical analysis would then be to use Gromov-Hausdorff limit, using convergent subsequence if necessary, but some technical issues made the method less appealing in some cases. Instead, the scale we define induces natural maps $X_n:\Omega_n\rightarrow Q$, where $Q$ is the Hilbert cube. Then, $\HOM$ can be naturally defined as the set of limit points of $X_n(\Omega_n).$ 

However, we find that considering $Q$ itself as our actual metric space of Theorem \ref{theorem_spectral_embedding} makes the proof even simpler, as the measure $\HMU$ on $\Borel(Q)$ naturally appears as a weak* limit point of the pushforward measures $X_n^{\#}(\mu_n)$. We can then show that $\HMU$ is supported on $\HOM$. We note that defining $\HMU$ this way relies on Banach-Alaoglu and Riesz-Markov Theorems.

From there, we find that the scale induces a natural isometry $\HU: H\rightarrow L_2(\HOM,\HMU).$ The rest is then the proof that the required square-integrable $\Hm:\HOM\rightarrow\mathbb{R}$ exists by using Riesz representation theorem for Hilbert space if necessary. One notes that with the further examples, where explicit solutions are desired, the use of subsequences and representation theorems will not be required.

\subsection{Conventions}

From now on, we assume that $H$ is a separable $\mathbb{K}$-Hilbert space, where $\mathbb{K}\in\{\mathbb{R}, \mathbb{C} \}.$ Furthermore, we assume that $A$ is a densely-defined symmetric operator, with no other properties unless stated.

We will note the left-linear inner product of $x$ and $y$ with $\LL x,y\RR$, while we note pairs of these elements with $(x,y). $ When $\mu$ is a measure on some set $\Omega$, and $h:\Omega\rightarrow\mathbb{K}$ is a square-integrable function, we may abuse the notation and also note $h$ as its $L_2(\Omega,\mu)$-class.

Considering $\overline{\mathbb{D}}=\{z\in\mathbb{K}\;|\; |z|\leq 1\},$ the closed unit disk of $\mathbb{K}$ , we note the space $Q=\overline{\mathbb{D}}^{\mathbb{N}}$ equipped with its usual product topology. Since $Q$ is known to be homeomorphic to $[0,1]^{\mathbb{N}}$ no matter the value of $\mathbb{K},$ we will refer to $Q$ as the Hilbert cube, a known compact metrizable space. We finally note the natural projections $\pi_j:Q\rightarrow\overline{\mathbb{D}}$ for $j\in\mathbb{N}.$
\section{Sampling-scale sequences}\label{section_sampling}

We now define the main objects concerning our process.
\begin{defn}\label{definition_sampling}
	A sequence of triplets $(H_n,A_n,\Omega_n)_{n\in\mathbb{N}}$ is called a \textbf{sampling-sequence} for $A$ if for any $n\in\mathbb{N}$,
	\begin{enumerate}
		\item\label{property_sampling_finite_dim} $H_n< H$ and $\operatorname{dim}(H_n)\in\mathbb{N}$;
		\item\label{property_sampling_symmetric_op} $A_n: H_n\rightarrow H_n$ is a symmetric linear operator;
		\item\label{property_sampling_eigenbasis} $\Omega_n$ is an orthonormal eigenbasis of $A_n$,
		\setcounter{pointnumber}{\value{enumi}}
	\end{enumerate}
	and if
	\begin{enumerate}
		\setcounter{enumi}{\value{pointnumber}}
		\item\label{property_sampling_approx} for any $x\in\Dom(A)$, there exists a sequence $(x_n)_{n\in\mathbb{N}}$ such that each $x_n\in H_n$, $\lim_{n\rightarrow\infty}x_n=x$ and $\lim_{n\rightarrow\infty} A_n x_n=Ax.$
		\setcounter{pointnumber}{\value{enumi}}
	\end{enumerate}
\end{defn}
\begin{prop}\label{proposition_sampling_existence}
	There exists a sampling sequence for $A$.
\end{prop}
\begin{proof}
	Since the graph $G(A)\subset H\times H$ is a separable space, let $(g_k)_{k\in\mathbb{N}}$ be a sequence in $\Dom(A)$ such that the sequence $(g_k, Ag_k)_{k\in\mathbb{N}}$ is dense in $G(A)$. Then, for $n\in\mathbb{N},$ let $H_n=\Span(\{g_k\}_{k=1}^n),$ a finite dimensional subspace of $\Dom(A)\subset H$, and let $A_n=(\Proj_{H_n}\circ A)|_{H_n}:H_n\rightarrow H_n.$ It is clear that $A_n$ is a linear operator on $H_n$, and symmetry holds since $$\LL A_nx, y \RR=\LL \Proj_{H_n}Ax,y\RR =\LL Ax,y\RR=\LL x, Ay\RR=\LL x, \Proj_{H_n} Ay\RR=\LL x, A_ny\RR$$ holds for any $x,y\in H_n.$ Thus, by the finite dimensional version of the Spectral Theorem, let $\Omega_n$ be some orthonormal eigenbasis of $A_n$. We show that $(H_n,A_n,\Omega_n)_{n\in\mathbb{N}}$ forms a sampling sequence for $A$. We already know that Properties (\ref{property_sampling_finite_dim}), (\ref{property_sampling_symmetric_op}) and (\ref{property_sampling_eigenbasis}) hold, so we only need to show (\ref{property_sampling_approx}).
	
	We note that $\{g_k\}_{k\in\mathbb{N}}$ is dense in $H$. Indeed, we have that for $x\in\Dom(A)$ and $k\in\mathbb{N}$, $\|x-g_k\|\leq\|(x,Ax)-(g_k,Ag_k)\|,$ and as such $\{g_k\}_{k\in\mathbb{N}}$ is dense in $\Dom(A)$, which itself is dense in $H$.
	
	  Let $x\in\Dom(A).$ Then, for $n\in\mathbb{N},$ let $x_n\in H_n$ be the element such that $(x_n, A_n x_n)=\Proj_{G(A_n)}(x,Ax).$ We show that $\lim_{n\rightarrow\infty} (x_n, A_nx_n)=(x,Ax).$ Let $\epsilon>0$, and let $k\in\mathbb{N}$ such that $\|(g_k,A g_k)-(x,Ax)\|<\frac{\epsilon}{2}.$ Then, let $l\in\mathbb{N}$ such that $\|A g_k-g_l\|<\frac{\epsilon}{2}.$ For any $n\geq l$, we have, since $g_l\in H_n,$  $$\|Ag_k-\Proj_{H_n}Ag_k\|\leq\|Ag_k-g_l\|<\frac{\epsilon}{2}.$$ 
	  
	  Therefore, if $N=\max\{l,k\}$, we have that for any $n>N$, 
	  \begin{align*}
	  	\|(x,Ax)-(x_n,A_nx_n)\|&\leq \|(x,Ax)-(g_k,A_ng_k)\|\\
	  	&\leq \|(x,Ax)-(g_k,Ag_k)\|+\|(g_k,Ag_k)-(g_k,A_ng_k)\|\\
	  	&=\|(x,Ax)-(g_k,Ag_k)\|+\|Ag_k-\Proj_{H_n}Ag_k\|<\epsilon.
	  \end{align*} 
	  
	  Thus, $\lim_{n\rightarrow\infty}(x_n,A_nx_n)=(x,Ax).$ Since $x\in\Dom(A)$ is arbitrary, we conclude that Property (\ref{property_sampling_approx}) holds, and $(H_n, A_n,\Omega_n)_{n\in\mathbb{N}}$ forms a sampling sequence for $A$.
\end{proof}
Next, we need another relatively peculiar object, the low-dimensional-biased scale.
\begin{defn}\label{definition_scale}
	The sequence of finite sequences of pairs $((e_j^{(n)},c_j^{(n)})_{j=1}^{N_n})_{n\in\mathbb{N}}$ is called a \textbf{low-dimensional-biased scale} if, for any $n\in\mathbb{N}$,
	\begin{enumerate}
		\setcounter{enumi}{\value{pointnumber}}
		\item\label{property_scale_elements} $N_n\in\mathbb{N}$, with $e_j^{(n)}\in H\setminus\{0\}$ and $c_j^{(n)}\in\mathbb{R}_{>0}$ for any $j\leq N_n$;
		\item\label{property_scale_proba} $\sum_{j=1}^{N_n}c_j^{(n)}\|e_j^{(n)}\|^2=1$,
		\setcounter{pointnumber}{\value{enumi}}
	\end{enumerate}
	and if
	\begin{enumerate}
		\setcounter{enumi}{\value{pointnumber}}
		\item\label{property_scale_infinity} $\lim_{n\rightarrow\infty}N_n=\infty$;
		\item\label{property_scale_limits}  for any $j\in\mathbb{N}$ we have $e_j:=\lim_{n\rightarrow\infty} e_j^{(n)}$ exists in $H\setminus\{0\}$ and $c_j:=\lim_{n\rightarrow\infty}c_j^{(n)}$ exists in $\mathbb{R}_{>0}$;
		\item\label{property_scale_dense} $\{e_j\}_{j\in\mathbb{N}}$ is dense-spanning in $H$;
		\item\label{property_scale_bias} $\sum_{j\in\mathbb{N}}c_j\|e_j\|^2=1.$
		\setcounter{pointnumber}{\value{enumi}}
	\end{enumerate}
\end{defn}

The idea of such scales is that they will allow to construct suitable measures on each $\Omega_n$, as well as metric structures on the $\Omega_n$ that allow interactions between each other. Thus, we will need a few more compatibility conditions.
\begin{defn}\label{definition_sampling_scale_sequence}
	The sequence $S=(S_n)_{n\in\mathbb{N}}=(H_n,A_n,\Omega_n,(e_j^{(n)},c_j^{(n)})_{j=1}^{N_n})_{n\in\mathbb{N}}$ is called an \textbf{$A$-sampling-scale} sequence if $(H_n,A_n,\Omega_n)_{n\in\mathbb{N}}$ forms a sampling sequence for $A$ while $((e_j^{(n)},c_j^{(n)})_{j=1}^{N_n})_{n\in\mathbb{N}}$ forms a low-dimensional-biased scale, and if
	\begin{enumerate}
		\setcounter{enumi}{\value{pointnumber}}
		\item\label{property_compat_H} for any $n\in\mathbb{N}$ and $j\leq N_n,$ $e_j^{(n)}\in H_n$;
		\item\label{property_compat_A} there exists $C\in\mathbb{R}$ such that for any $n\in\mathbb{N}$, $\sum_{j=1}^{N_n}c_j^{(n)}\|A_ne_j^{(n)}\|^2<C$;
		\item\label{property_compat_Omega} for any $n\in\mathbb{N}$ and $f\in\Omega_n,$ there exists $j\leq N_n$ such that $\langle e_j^{(n)},f \rangle \neq 0.$
		\setcounter{pointnumber}{\value{enumi}}
	\end{enumerate}
	
	Finally, the $A$-sampling-scale sequence is called \textbf{strong} if 
	\begin{enumerate}
		\setcounter{enumi}{\value{pointnumber}}
		\item\label{property_compat_strong} For any $j\in\mathbb{N},$ the sequence $(A_ne_j^{(n)})_{n\in\mathbb{N}}$ converges.
	\end{enumerate}	
\end{defn}
\begin{remark}
	One can show directly that for any $A$-sampling-scale sequence, $e_j\in\Dom(A^*)$ and that $A^* e_j$ is the weak limit of $(A_n e_j^{(n)})_{n\in\mathbb{N}}$, hence the term "strong" for Property (\ref{property_compat_strong}). We will not use that here, and so we omit the proof.
\end{remark}
We end this section by showing that not only do $A$-sampling-scale sequences exist, but that a suitable scale exists given any sampling sequence.

\begin{prop}\label{proposition_scale_existence}
	For any sampling sequence $(H_n,A_n,\Omega_n)_{n\in\mathbb{N}}$ for $A$, there exists a scale $((e_j^{(n)},c_j^{(n)})_{j=1}^{N_n})_{n\in\mathbb{N}}$ such that $(S_n)_{n\in\mathbb{N}}=(H_n,A_n,\Omega_n,(e_j^{(n)},c_j^{(n)})_{j=1}^{N_n})_{n\in\mathbb{N}}$ is a strong $A$-sampling-scale sequence.
\end{prop}
\begin{proof}
	Let $(H_n,A_n,\Omega_n)_{n\in\mathbb{N}}$ be an arbitrary sampling sequence for $A$.
	
	Let $(g_j)_{j\in\mathbb{N}}$ be a Hilbert basis of $H$ such that for all $j\in\mathbb{N}$, $g_j\in\operatorname{dom}(A)$. It exists, since $H$ is separable and $\operatorname{dom}(A)$ is dense. Given $j\in\mathbb{N}$, let $(g_j^{(n)})_{n\in\mathbb{N}}$ be a sequence given by Property (\ref{property_sampling_approx}) of Definition \ref{definition_sampling} used on $g_j$. As such, $g_j^{(n)}\in H_n$, $g_j=\lim_{n\rightarrow\infty}g_j^{(n)}$ and $Ag_j=\lim_{n\rightarrow\infty}A_ng_j^{(n)}$. Since $\dim(H_n)>0$ for each $n$, and since $g_j\neq 0$ for each $j$, we can assume without loss of generality that $g_j^{(n)}$ is always nonzero. 
	
	For $n\in\mathbb{N}$, let $N_n:=n+1$. For $j\in[N_n]$, let $$e_j^{(n)}=\begin{cases}g_j^{(n)} \quad \text{ if } j\leq n\\\sum_{f\in\Omega_n}f \quad \text{ otherwise. }\end{cases}$$ 
	Furthermore, for such $j$, let 
	\begin{align*}
		a_j^{(n)}&=\frac{1}{2^j\max\left(\|e_j^{(n)}\|^2,\|A_n e_j^{(n)}\|^2\right)}\\
		c_j^{(n)}&=\frac{1}{\sum_{l=1}^{N_n}a_l^{(n)}\|e_l^{(n)}\|^2}a_j^{(n)}.	
	\end{align*}

	We first establish that $((e_j^{(n)},c_j^{(n)})_{j=1}^{N_n})_{n\in\mathbb{N}}$ is a low-dimensional-biased scale. Properties (\ref{property_scale_elements}), (\ref{property_scale_proba}) and (\ref{property_scale_infinity}) are directly verified. Furthermore, it is clear that for $j\in\mathbb{N},$ \begin{align*}e_j=\lim_{n\rightarrow\infty}e_j^{(n)}&=\lim_{n\rightarrow\infty}g_j^{(n)}=g_j,\\ \lim_{n\rightarrow\infty}A_ne_j^{(n)}&=\lim_{n\rightarrow\infty}A_ng_j^{(n)}=Ag_j,
	\end{align*} 
	and that $\lim_{n\rightarrow\infty}a_j^{(n)}=\frac{1}{2^j\max(1,\|Ag_j\|^2)}=:a_j.$ Since $0<a_j^{(n)}\|e_j^{(n)}\|^2\leq\frac{1}{2^j}$, we use the Dominated Convergence Theorem to conclude that $$c_j=\lim_{n\rightarrow\infty}c_j^{(n)}=\frac{1}{\sum_{l\in\mathbb{N}}a_l\|g_l\|^2}a_j>0.$$ Thus, Property (\ref{property_scale_limits}) holds. $e_j=g_j$ ensures that Property (\ref{property_scale_dense}) holds, while Property (\ref{property_scale_bias}) follows from that last equality. Thus, $((e_j^{(n)},c_j^{(n)})_{j=1}^{N_n})_{n\in\mathbb{N}}$ is a low-dimensional biased scale.
	
	We now show that $(S_n)_{n\in\mathbb{N}}$ is a strong $A$-sampling-scale sequence. Property (\ref{property_compat_H}) follows from the definition of $e_j^{(n)}$ and Property (\ref{property_compat_Omega}) follows from the fact that for any $f\in\Omega_n,$ $\LL e_j^{n+1}, f\RR=1.$ Finally, we have that  for any $n\in\mathbb{N},$
	\begin{align*}
		\sum_{j=1}^{N_n} c_j^{(n)}\|A_ne_j^{(n)}\|^2&=\frac{1}{\sum_{l=1}^{N_n}a_l^{(n)}\|e_l^{(n)}\|^2}\sum_{j=1}^{N_n}a_j^{(n)}\|A_ne_j^{(n)}\|^2\\
		&\leq \frac{1}{a_1^{(n)}\|e_1^{(n)}\|^2}\sum_{j=1}^{N_n}\frac{1}{2^j}<\frac{1}{a_1^{(n)}\|e_1^{(n)}\|^2}.
	\end{align*}
	That last sequence converges to $\frac{1}{a_1\|e_1\|^2}$ and as such is bounded above, thus Property (\ref{property_compat_A}) holds. Therefore, $(S_n)_{n\in\mathbb{N}}$ is an $A$-sampling-scale sequence. Since we have already established that $\lim_{n\rightarrow\infty} A_ne_j^{(n)}=A g_j$ for any $j\in\mathbb{N}$, we conclude that Property (\ref{property_compat_strong}) holds and $(S_n)_{n\in\mathbb{N}}$ is strong.
\end{proof}
\begin{remark}\label{remark_sampling_scale_subsequences}
	Given $S=(S_n)_{n\in\mathbb{N}}$, an $A$-sampling-scale sequence, it is direct to show that any subsequence $S'=(S_{n_m})_{m\in\mathbb{N}}$ is also an $A$-sampling-scale sequence.
\end{remark}
\section{The finite dimensional isometries}

We now fix the (not necessarily strong) $A$-sampling-scale sequence $$S=(S_n)_{n\in\mathbb{N}}=(H_n,A_n,\Omega_n,(e_j^{(n)},c_j^{(n)})_{j=1}^{N_n})_{n\in\mathbb{N}}.$$ We also denote the eigenvalue function of $A_n$ by $\lambda_n:\Omega_n\rightarrow \mathbb{R}.$
\begin{definition}
	For $n\in\mathbb{N}$, let $\mu_n:\mathcal{P}(\Omega_n)\rightarrow\mathbb{R}_{\geq 0}$ with $$\mu_n(V)=\sum_{j=1}^{N_n} c_j^{(n)}\|\Proj_{\Span(V)}e_j^{(n)}\|^2=\sum_{f\in V}\sum_{j=1}^{N_n}c_j^{(n)}|\langle e_j^{(n)}, f\rangle|^2. $$
\end{definition}
\begin{remark}
	For $f\in\Omega_n$, we also note $\mu_n(f):=\mu_n(\{f\})=\sum_{j=1}^{N_n}c_j^{(n)}|\langle e_j^{(n)}, f\rangle|^2.$
\end{remark}

The following is a direct result from Properties (\ref{property_scale_proba}) and (\ref{property_compat_Omega}).
\begin{prop}
	We have that $\mu_n$ is a probability measure on $\mathcal{P}(\Omega_n)$. Furthermore, $\mu_n(f)>0$ holds for any $f\in\Omega_n$.
\end{prop}

Thus, as vector spaces, we have that $L_2(\Omega_n,\mu_n)=\mathbb{K}^{\Omega_n}$. This allows us to define the following unitary maps.

\begin{definition}
	For $n\in\mathbb{N},$ let $U_n:H_n\rightarrow L_2(\Omega_n,\mu_n)$ be given by $$(U_n(x))(f)=\frac{\LL x, f\RR}{\sqrt{\mu_n(f)}}.$$
\end{definition}
\begin{prop}\label{proposition_finite_isometry}
	For any $n\in\mathbb{N}$, $U_n$ is a unitary map. Furthermore, for any given $x\in H_n$, we have $\lambda_n\cdot U_n(x)=U_n(A_nx).$ Finally, $|(U_n(e_j^{(n)}))(f)|^2\leq \frac{1}{c_j^{(n)}}$ whenever $j\leq N_n$ and $f\in\Omega_n.$
\end{prop}
\begin{proof}
	Linearity of $U_n$ is a direct result from its definition. Furthermore, for any $x\in H_n$, 
	\begin{align*}
		\|U_n(x)\|^2&=\int_{\Omega_n}|U_n(x)|^2d\mu_n=\sum_{f\in\Omega_n}\left|\frac{\LL x, f\RR}{\sqrt{\mu_n(f)}}\right|^2\mu_n(f)=\sum_{f\in\Omega_n}|\LL x,f\RR|^2=\|x\|^2.
	\end{align*}
	Thus, as a norm-preserving linear map, $U_n$ is an isometry. Since $\dim(H_n)=|\Omega_n|,$ $U_n$ is unitary.
	
	For the second part, let $x\in H_n.$ We note, for any $f\in\Omega_n$, that $$(\lambda_n\cdot U_n(x))(f)=\lambda_n(f)\frac{\LL x,f\RR}{\sqrt{\mu_n(f)}}=\frac{\LL x, A_n f\RR}{\sqrt{\mu_n(f)}}=\frac{\LL A_nx, f\RR}{\sqrt{\mu_n(f)}}=(U_n(A_nx))(f). $$
	
	Thus, $\lambda_n\cdot U_n(x)=U_n(A_nx).$ Finally, let $j\leq N_n,$ and let $f\in\Omega_n.$ We have
	\begin{align*}
		|(U_n(e_j^{(n)}))(f)|^2&=\frac{1}{c_j^{(n)}}\frac{c_j^{(n)}|\LL e_j^{(n)}, f\RR|^2}{\mu_n(f)}\leq \frac{1}{c_j^{(n)}},
	\end{align*}
	proving the inequality.
\end{proof}

We can now define our sequence of embeddings, as described in Section \ref{section_idea}.

\begin{definition}\label{definition_finite_embedding_cube}
	For $n\in\mathbb{N},$ let $X_n:\Omega_n\rightarrow Q$ be defined by $$\pi_j\circ X_n=\begin{cases}
		\sqrt{c_j^{(n)}}U_n(e_j^{(n)})\quad\text{ if $j\leq N_n$ }\\
		0\quad \text{ otherwise }
	\end{cases}$$ for $j\in\mathbb{N}.$
\end{definition}

We can now define the limit space $\HOM$ as well.

\begin{definition}\label{definition_limit_space}
	We define $\HOM\subset Q$ with $$\HOM=\bigcap_{n=1}^{\infty}\overline{\bigcup_{m=n}^{\infty}X_m(\Omega_m)}.$$ In other words, $x\in\hat{\Omega}$ if and only if there exists a subsequence $(\Omega_{n_m})_{m\in\mathbb{N}}$ and a sequence $(f_m\in\Omega_{n_m})_{n\in\mathbb{N}}$ such that $x=\lim_{m\rightarrow\infty}X_{n_m}(f_m).$
\end{definition}
\begin{remark}
	It is clear that $\HOM$ is a compact subset of $Q$. If $X_n(\Omega_n)$ converges under Hausdorff metric, then $\HOM$ is the limit. In fact, $\HOM$ is the union of all Hausdorff limit points of $(X_n(\Omega_n))_{n\in\mathbb{N}}.$
\end{remark}

We now define the notion of convergence for sampling-scale sequences.

\begin{definition}
	We say that the $A$-sampling-scale sequence $S$ is \textbf{convergent} if for any continuous $h:Q\rightarrow\mathbb{K},$ the sequence $$\left(\int_{\Omega_n} h\circ X_n d\mu_n\right)_{n\in\mathbb{N}}$$ converges.
\end{definition}
\begin{remark}\label{remark_weak_star_equivalence}
	For $n\in\mathbb{N},$ let $\HMU_n=X_n^{\#}(\mu_n)$ be the push-forward probability measure of $\mu_n$ on $\Borel(Q)$ through $X_n$. Then, $S$ is convergent if and only if $(\HMU_n)_{n\in\mathbb{N}}$ is weak* convergent in $C(Q)'$.
\end{remark}
\begin{prop}
	Any $A$-sampling-scale sequence $S=(S_n)_{n\in\mathbb{N}}$ has a convergent subsequence $S'=(S_{n_m})_{m\in\mathbb{N}}$. In particular, there always exists a convergent $A$-sampling-scale sequence.
\end{prop}
\begin{proof}
	Since $Q$ is a compact metrizable space, $C(Q)$ is a separable Banach space. By Alaoglu's Theorem, the unit ball of $C(Q)'$ is a metrizable compact space under weak* topology. 
	
	For $n\in\mathbb{N},$ we define $\HMU_n$ as per Remark \ref{remark_weak_star_equivalence}. Then, since each $\HMU_n$ is a probability measure, we have that $\|\HMU_n\|=1$ in $C(\HOM)'$. Therefore, there exists a weak* convergent subsequence $(\HMU_{n_m})_{m\in\mathbb{N}}$. Using Remark \ref{remark_weak_star_equivalence}, we conclude that  $S'=(S_{n_m})_{m\in\mathbb{N}}$ is convergent.
\end{proof}
\section{Resulting probability measure}
We now assume that the fixed $A$-sampling-scale sequence $$S=(S_n)_{n\in\mathbb{N}}=(H_n,A_n,\Omega_n,(e_j^{(n)},c_j^{(n)})_{j=1}^{N_n})_{n\in\mathbb{N}}$$ is convergent. We now define the measure $\HMU$ on $\Borel(Q)$ using the following, which is a consequence of the Riesz-Markov-Kakutani representation theorem.

\begin{prop}\label{proposition_riesz_markov}
	There exists a unique probability measure $\HMU:\Borel(Q)\rightarrow\mathbb{R}_{\geq 0}$ such that for any continuous function $h:Q\rightarrow\mathbb{K}$, $$\int_{Q}h d\HMU=\lim_{n\rightarrow\infty}\int_{\Omega_n}(h\circ X_n)d\mu_n.$$
\end{prop}

In general, one cannot calculate $\HMU(V)$ directly as the limit of $\HMU_n(V)=\mu_n(X_n^{-1}(V))$, as $\HMU$ can have a disjoint support from all the finite measures. This is why the use of continuous functions was required to define convergence. The next proposition, however, is useful for such direct calculations of $\HMU$ using values of $\mu_n$.
\begin{prop}\label{proposition_open_measure}
	If $V$ is an open subset of $Q$, then $\HMU(V)\leq \liminf_{n\rightarrow\infty} \mu_n(X_n^{-1}(V)).$
\end{prop}
\begin{proof}
	Let $V$ be an open subset of $Q$. Since $Q$ is a compact metric space and $\HMU$ is a probability measure on its Borel sets, we have that $\HMU$ is Radon. Thus, we know $\HMU(V)=\sup(\{\HMU(K)\;|\; K \text{ is a compact subset of $V$ } \}).$ However, since $Q$ is a metric space, for any compact subset $K$ of $V$, there exists a continuous $h:Q\rightarrow [0,1]$ supported in $V$ such that $h|_K =1$, thus $\HMU(K)\leq \int_{Q} h d\HMU \leq \HMU(V).$ Therefore, we have $$\HMU(V)=\sup(\{\int_{Q}hd\HMU\;|\; h:Q\rightarrow[0,1] \text{ is continuous and supported in $V$ } \}).$$
	
	For any such continuous $h$ supported in $V$ with values in $[0,1]$, we have 
	\begin{align*}
		\int_{Q}hd\HMU=\lim_{n\rightarrow\infty}\int_{\Omega_n}(h\circ X_n)d\mu_n\leq\liminf_{n\rightarrow\infty}\int_{\Omega_n}(\mathbf{1}_V\circ X_n)d\mu_n=\liminf_{n\rightarrow\infty} \mu_n(X_n^{-1}(V)).
	\end{align*}
	Thus, by definition of $\sup$, $\HMU(V)\leq\liminf_{n\rightarrow\infty} \mu_n(X_n^{-1}(V)).$
\end{proof}
\begin{cor}\label{corollary_suport_limits}
	We have that $\HMU$ is supported on $\HOM$. In other words, $\HMU(Q\setminus\HOM)=0.$
\end{cor}
\begin{proof}
	Since $\HOM=\bigcap_{k=1}^\infty\overline{\bigcup_{m=k}^{\infty}X_m(\Omega_m)},$ it is sufficient to establish that for any $k$, $\HMU(Q\setminus\overline{\bigcup_{m=k}^{\infty}X_m(\Omega_m)})=0.$ We have, for any $n\geq k,$ that $X_n^{-1}\left(Q\setminus\overline{\bigcup_{m=k}^{\infty}X_m(\Omega_m)}\right)$ is empty, thus $\liminf_{n\rightarrow\infty} \mu_n\left(X_n^{-1}\left(Q\setminus\overline{\bigcup_{m=k}^{\infty}X_m(\Omega_m)}\right)\right)=0.$  
	
	Since $Q\setminus\overline{\bigcup_{m=k}^{\infty}X_m(\Omega_m)}$ is open, we have $\HMU\left(X_n^{-1}\left(Q\setminus\overline{\bigcup_{m=k}^{\infty}X_m(\Omega_m)}\right)\right)=0$ from Proposition \ref{proposition_open_measure}.
\end{proof}
\begin{remark}
	Corollary \ref{corollary_suport_limits} allows us to identify $L_2(\HOM,\HMU)=L_2(Q,\HMU),$ with all relevant objects of Theorem \ref{theorem_spectral_embedding} restricting naturally on $\HOM$ as long as they have been suitably defined on $Q$. However, it remains useful to work in $Q$ to first establish such definitions.
\end{remark}

We end this section with another consequence of Proposition \ref{proposition_open_measure}, which we will use in the next section, after having proven the following lemma. 

\begin{lemma}\label{lemma_uniform_series}
	For any $\epsilon>0,$ there exists $J_{\epsilon}\in\mathbb{N}$ such that for any $n\in\mathbb{N},$ $$\sum_{J_{\epsilon}<j\leq N_n}c_j^{(n)}\|e_j^{(n)}\|^2<\epsilon.$$
\end{lemma}
\begin{proof}
	Let $\epsilon>0.$ By Property (\ref{property_scale_bias}), let $J_0\in\mathbb{N}$ such that $\sum_{j=1}^{J_0}c_j\|e_j\|^2>1-\frac{\epsilon}{2}.$ Then, using Properties (\ref{property_scale_infinity}) and  (\ref{property_scale_limits}) on each $j\leq J_0$, let $K\in\mathbb{N}$ such that for any $n>K$, we have $N_n\geq J_0$ and  $c_j^{(n)}\|e_j^{(n)}\|^2>c_j\|e_j\|^2-\frac{\epsilon}{2J_0}$ whenever $j\leq J_0.$ For any such $n$, we then have $\sum_{j=1}^{J_0} c_j^{(n)}\|e_j^{(n)}\|^2>1-\epsilon,$ thus $\sum_{j=J_0+1}^{N_n}c_j^{(n)}\|e_j^{(n)}\|^2<\epsilon$ by Property (\ref{property_scale_proba}). Let $J_{\epsilon}=\max(\{N_n\}_{n=1}^K\cup \{J_0\}).$ Then, for any $n\in\mathbb{N},$ we have $$\sum_{J_{\epsilon}<j\leq N_n}c_j^{(n)}\|e_j^{(n)}\|^2\leq\begin{cases}
		\sum_{j=J_0+1}^{N_n}c_j^{(n)}\|e_j^{(n)}\|^2&<\epsilon\quad \text{ if $n>K$ }	\\
		0&<\epsilon\quad \text{ otherwise, }
		\end{cases}$$ completing the proof.
\end{proof}

\begin{prop}\label{proposition_null_origin}
	Given $O=(0)_{n\in\mathbb{N}}\in Q,$ $\HMU(\{O\})=0.$
\end{prop}
\begin{proof}
		For $k,l\in\mathbb{N},$ let $E_{k,l}=\bigcap_{j=1}^l \pi_j^{-1}(D_{\frac{1}{k}})\subset Q,$ where $D_{\frac{1}{k}}$ is the open disk of radius $\frac{1}{k}$ in $\mathbb{K}$ centered at $0$. By definition of the product topology, we have that $E_{k,l}$ is an open subset of $Q$ containing $O$ for any $k,l$.  Thus, using proposition \ref{proposition_open_measure}, we always have $\HMU(\{O\})\leq \HMU(E_{k,l})\leq \liminf_{n\rightarrow\infty} \mu_n(X_n^{-1}(E_{k,l})).$
		
		Furthermore, upon inspection, we have that for any such $k,l\in \mathbb{N},$  $$X_n^{-1}(E_{k,l})=\bigcap_{j=1}^l (\pi_j\circ X_n)^{-1}(D_{\frac{1}{k}})=\bigcap_{j=1}^{\min(l,N_n)} U_n(e_j^{(n)})^{-1}(D_{\frac{1}{k\sqrt{c_j^{(n)}}}}),$$ and so $|\LL e_j^{(n)}, f\RR|^2 < \frac{\mu_n(f)}{k^2 c_j^{(n)}}$ whenever $f\in X_n^{-1}(E_{k,l})$ and $j\leq \min(l,N_n).$ 
		
		Let $\epsilon>0$, and per Lemma \ref{lemma_uniform_series} let $J_{\epsilon}\in\mathbb{N}$ such that $\sum_{J_{\epsilon}<j\leq N_n}c_j^{(n)}\|e_j^{(n)}\|^2<\epsilon$ for any $n\in \mathbb{N}.$ We have, for any $n$ such that $N_n>J_{\epsilon}$ (which is guaranteed for large enough $n$ by Property (\ref{property_scale_infinity}) of Definition \ref{definition_scale}) and any $k\in\mathbb{N}$:
		\begin{align*}
				\mu_n(X_n^{-1}(E_{k,J_{\epsilon}}))&=\sum_{j=1}^{N_n}c_j^{(n)}\sum_{f\in X_n^{-1}(E_{k,J_{\epsilon}})} |\LL e_j^{(n)}, f\RR|^2\\
				&\leq \sum_{j=1}^{J_{\epsilon}}c_j^{(n)}\sum_{f\in X_n^{-1}(E_{k,J_{\epsilon}})} |\LL e_j^{(n)}, f\RR|^2+\sum_{J_{\epsilon}<j\leq N_n}c_j^{(n)}\|e_j^{(n)}\|^2\\
				&\leq \epsilon+\sum_{j=1}^{J_{\epsilon}}c_j^{(n)}\sum_{f\in X_n^{-1}(E_{k,J_{\epsilon}})} |\LL e_j^{(n)}, f\RR|^2\\
				&\leq\epsilon+\sum_{j=1}^{J_{\epsilon}}c_j^{(n)}\sum_{f\in X_n^{-1}(E_{k,J_{\epsilon}})}\frac{1}{k^2c_j^{(n)}}\mu_n(f)\\
				&=\epsilon+\frac{J_{\epsilon}}{k^2}\mu_n(X_n^{-1}(E_{k,J_{\epsilon}}))\leq \epsilon+\frac{J_{\epsilon}}{k^2}. 
			\end{align*}  
		
		Thus, we have that $\HMU(\{O\})\leq\liminf_{n\rightarrow\infty} \mu_n(X_n^{-1}(E_{k,J_{\epsilon}}))\leq \epsilon+\frac{J_{\epsilon}}{k^2}.$ Since $k$ is arbitrary, we have that $\HMU(\{O\})\leq\epsilon.$ Since $\epsilon>0$ is arbitrary, we have that $\HMU(\{O\})=0.$
\end{proof}

\section{Induced Isometry}
As with the previous section, we assume that the fixed $A$-sampling-scale sequence $$S=(S_n)_{n\in\mathbb{N}}=(H_n,A_n,\Omega_n,(e_j^{(n)},c_j^{(n)})_{j=1}^{N_n})_{n\in\mathbb{N}}$$ is convergent, with the probability measure $\HMU$ on $\Borel(Q)$ defined the same way. The goal of this section is to introduce a suitable isometry $\HU:H\rightarrow L_2(Q,\HMU)$ and real-valued function $\Hm$ satisfying Theorem \ref{theorem_spectral_embedding}, completing the process.
 
We now define an important sequence of continuous functions on $Q$.

\begin{definition}
	For $j\in\mathbb{N},$ let $V_j: Q\rightarrow\mathbb{K}$ be given by  $V_j=\frac{1}{\sqrt{c_j}}\pi_j.$
\end{definition}
\begin{remark}
	It is clear that for each $j$, $V_j$ is continuous on $Q$. Furthermore, as $Q$ is compact and $\HMU$ is a probability measure, it is also clear that $V_j$ is square integrable. Finally, we note that for any $n\in\mathbb{N}$ and $j\leq N_n$, $V_j\circ X_n=\sqrt{\frac{c_j^{(n)}}{c_j}}U_n(e_j^{(n)}).$
\end{remark}
We can now define the natural isometry $\HU: H\rightarrow L_2(Q,\HMU)$ as well, which is done using the following proposition.

\begin{prop}\label{proposition_unique_isometry}
	There is a unique linear isometry $\HU: H\rightarrow L_2(Q,\HMU)$ such that for any $j\in\mathbb{N}$, $\HU(e_j)=V_j.$
\end{prop}
\begin{proof}
	The uniqueness of such an isometry holds by density of $\Span(\{e_j\}_{j\in\mathbb{N}})$, which is established by Property (\ref{property_scale_dense}) of Definition \ref{definition_scale}. Therefore, we only need to show the existence of such $\HU$.
	
	The most important part is proving that for any $j,l\in\mathbb{N},$ $\LL V_j,V_l\RR=\LL e_j, e_l\RR.$ We calculate that for such $j$ and $l$, since $V_j\overline{V_l}$ is continuous on $Q$:
	
	\begin{align*}
		\LL V_j, V_l\RR&=\int_{Q} V_j\overline{V_l}d\HMU=\lim_{n\rightarrow\infty} \int_{\Omega_n} \left(V_j\overline{V_l}\right)\circ X_n d\mu_n\\
		&=\lim_{n\rightarrow\infty}\sqrt{\frac{c_j^{(n)}}{c_j}\frac{c_l^{(n)} }{c_l}}\int_{\Omega_n}U_n\left(e_j^{(n)}\right)\overline{U_n\left(e_l^{(n)}\right)}d\mu_n\\
		&=\lim_{n\rightarrow\infty}\left\LL U_n\left(e_j^{(n)}\right), U_n\left(e_l^{(n)}\right)\right\RR=\lim_{n\rightarrow\infty}\LL e_j^{(n)}, e_l^{(n)}\RR\\
		&=\LL e_j, e_l\RR.
	\end{align*}
	
	The last equality, as well as the fourth, are due to Property (\ref{property_scale_limits}) of Definition \ref{definition_scale}. Thus, we have that for any $a_1,\dots, a_k \in \mathbb{K},$ $\|\sum_{j=1}^{k}a_j e_j\|^2=\|\sum_{j=1}^{k}a_j V_j\|^2.$ Therefore, there is a well-defined linear	 isometry $\HU:\Span(\{e_j\}_{j\in\mathbb{N}})\rightarrow L_2(Q,\HMU)$ such that $\HU(\sum_{j=1}^{k} a_j e_j)=\sum_{j=1}^{k} a_j V_j$, which extends to $\overline{\Span(\{e_j\}_{j\in\mathbb{N}})}=H$. 
\end{proof}

To complete the intended process, we now only need to define the suitable $\Hm$, which is done through the following theorem.
\begin{theorem}\label{theorem_multiplication_operator}
	There exists a unique real-valued $\Hm \in L_2(Q, \HMU)$ such that for any $x\in\Dom(A)$, $\HU(Ax)=\Hm\cdot \HU(x)$ as elements of $L_2(Q,\HMU).$
\end{theorem}

The rest of this section will be dedicated to proving this theorem. First, we need a few tools.

We want to be able to construct $\Hm$ as the "limit" of the eigenvalue functions $\lambda_n$, and do it in such a way that $\lambda_n\cdot U_n(x_n)$ "converges" to $\Hm\cdot\HU(x)$ whenever $x_n\in H_n$ converges to $x\in\Dom(A)$, with $U_n(x_n)$ itself "converging" to $\HU(x).$ To hope to have that, we need to define a proper convergence notion.

\begin{definition}\label{definition_L2_limit}
	For $n\in\mathbb{N},$ let $\Phi_n:L_2(\Omega_n,\mu_n)\rightarrow C(Q)'$, with $$(\Phi_n(g_n))(h)=\int_{\Omega_n} (h\circ X_n)g_nd\mu_n.$$ Furthermore, we say that a bounded sequence $(g_n\in L_2(\Omega_n,\mu_n))_{n\in\mathbb{N}}$ $S$-converges if $(\Phi_n(g_n))_{n\in\mathbb{N}}$ converges in $C(Q)'$ in the weak* topology. In other words,  $(g_n)_{n\in\mathbb{N}}$ $S$-converges if the $L_2$ norms are uniformly bounded, and if $\left(\int_{\Omega_n}(h\circ X_n)g_nd\mu_n\right)_{n\in\mathbb{N}}$ converges for any $h\in C(Q).$
\end{definition}
\begin{prop}\label{proposition_L2_limit}
	For any $S$-convergent sequence $(g_n\in L_2(\Omega_n,\mu_n))_{n\in\mathbb{N}}$, there exists a unique $g\in L_2(Q,\HMU)$ such that for any $h\in C(Q),$ $\int_{Q}hgd\HMU=\lim_{n\rightarrow\infty}(\Phi_n(g_n))(h).$ Furthermore, $\|g\|\leq\liminf_{n\rightarrow\infty}\|g_n\|$. Finally, if each $g_n$ is real-valued, so is $g$.
\end{prop}
\begin{proof}
	Let $\iota:C(Q)\rightarrow L_2(Q,\HMU)$ be the natural inclusion, which has dense image. Furthermore, let $M=\liminf_{n\rightarrow\infty}\|g_n\|\in\mathbb{R}.$ We want to define $\Phi:\iota(C(Q))\rightarrow\mathbb{K}$ with $\Phi(\iota(h))=\lim_{n\rightarrow\infty} (\Phi_n(g_n))(h).$ To show that such a definition is possible, we note, for any $h\in C(Q),$
	\begin{align*}
		\left|\int_{\Omega_n}(h\circ X_n)g_nd\mu_n\right|^2\leq\int_{\Omega_n}|g_n|^2d\mu_n \int_{\Omega_n}|h\circ X_n|^2d\mu_n= \|g_n\|^2 \int_{\Omega_n}|h|^2\circ X_nd\mu_n.
	\end{align*}
	Therefore, $|\lim_{n\rightarrow\infty}(\Phi_n(g_n))(h)|^2\leq M^2\int_{Q}|h|^2d\HMU=M\|\iota(h)\|^2,$ and so we have that $\lim_{n\rightarrow\infty}(\Phi_n(g_n))(h)=0$ whenever $\iota(h)=0.$ From that, we conclude that $\Phi$ is indeed a well-defined bounded linear relation on $\iota(C(Q))$ with $\|\Phi\|\leq M.$ Since $\iota(C(Q))$ is dense in $L_2(Q,\HMU),$ Riesz representation theorem for Hilbert space ensures that there is a unique $g\in L_2(Q,\HMU)$ such that for any $h\in C(Q),$ $\Phi(\iota(h))=\LL \iota(h), \overline{g}\RR$, thus $$\lim_{n\rightarrow\infty}(\Phi_n(g_n))(h)=\int_{Q}h gd\HMU.$$ The theorem further ensures that $\|g\|=\|\overline{g}\|=\|\Phi\|\leq M.$
	
	Finally, if each $g_n$ is real valued, then for any $h\in C(Q)$, we have $\overline{h}$ is also continuous, and
	\begin{align*}
		\LL \iota(h), \overline{g}\RR&= \overline{\LL \iota(\overline{h}), g\RR}=\lim_{n\rightarrow\infty}\overline{\int_{\Omega_n} (\overline{h}\circ X_n)g_nd\mu_n}=\lim_{n\rightarrow\infty}(\Phi_n(g_n))(h)=\LL \iota(h),g\RR.
	\end{align*}
	Thus, by density of $\iota(C(Q))$, $g=\Re(g)$ is real-valued.
\end{proof}
\begin{definition}
	If $(g_n\in L_2(\Omega_n,\mu_n))_{n\in\mathbb{N}}$ $S$-converges, we say that $\lim_{n\rightarrow\infty}^{(S)} g_n \in L_2(Q,\HMU)$ is the unique element given by the previous proposition.
\end{definition}
\begin{remark}\label{remark_continuous_convergence}
	By definition of convergence of $A$-sampling-scale sequence, we know that $\lim^{(S)}_{n\rightarrow\infty}\mathbf{1}_{\Omega_n}=1_{Q}.$ In fact, for any $h_0\in C(Q)$, $h_0=\lim^{(S)}_{n\rightarrow\infty}h_0\circ X_n,$ noting that $\int_{\Omega_n}|h_0\circ X_n|^2d\mu_n\rightarrow \int_{Q}|h_0|^2d\HMU,$ ensuring the sequence $(h_0\circ X_n)_{n\in\mathbb{N}}$ is bounded.
\end{remark}

We can now state a key property, that if $(x_n\in H_n)_{n\in\mathbb{N}}$ is a sequence for which $x=\lim_{n\rightarrow\infty} x_n$ in $H$, then $\HU(x)=\lim^{(S)}_{n\rightarrow\infty} U_n(x_n).$  Since we are going to work with $\lambda_n\cdot U_n(x_n),$ we will say a bit more.

\begin{prop}\label{proposition_converging_isometry}
	Suppose $(g_n\in L_2(\Omega_n,\mu_n))_{n\in\mathbb{N}}$ is a $S$-convergent sequence with $g=\lim^{(S)}_{n\rightarrow\infty}g_n,$ and that $(x_n\in H_n)_{n\in\mathbb{N}}$ converges to $x\in H.$ Suppose further that $(g_n\cdot U_n(x_n))_{n\in\mathbb{N}}$ is bounded (as respective elements of $L_2(\Omega_n,\mu_n)$). Then, we have $g\cdot\HU(x)\in L_2(Q,\HMU),$ and $g\cdot \HU(x)=\lim^{(S)}_{n\rightarrow\infty} g_n\cdot U_n(x_n).$ 
\end{prop}
\begin{proof}
	For this proof, since norms may present some ambiguity, we will use $\|\; \|_0$ to denote the $\sup$ norm on bounded functions, while $\|\; \|_{\mu}$ will denote the $L_2$ norm with respect to measure $\mu$.
	
	We first show that $\Phi_n(g_n\cdot U_n(x_n))$ weak* converges, with $$\lim_{n\rightarrow\infty} (\Phi_n(g_n\cdot U_n(x_n)))(h)=\lim_{n\rightarrow\infty}\int_{\Omega_n} (h\circ X_n)g_n\cdot U_n(x_n)d\mu_n=\int_{Q}hg\cdot\HU(x)d\HMU$$ for any $h\in C(Q)$. We now fix such $h$. Also, we fix $M\in\mathbb{R}$ such that $\|g_n\|_{\mu_n}\leq M$ for all $n$. We note that by Proposition \ref{proposition_L2_limit}, $\|g\|_{\HMU}\leq M$ as well.
	
	Let $\epsilon>0$.  By Property \ref{property_scale_dense} of Definition \ref{definition_scale}, let $\{a_j \}_{j=1}^k\subset\mathbb{K}$ with  $\left\|\sum_{j=1}^ka_je_j-x \right\|<\frac{\epsilon}{4M\|h\|_0+1}$. Then, let $N\in\mathbb{N}$ such that for any $n>N$, $N_n>k$ and:
	\begin{itemize}
		\item $\left|\int_{Q}g\cdot\left(h\sum_{j=1}^ka_jV_j\right)d\HMU-\int_{\Omega_n}g_n\cdot\left(h\sum_{j=1}^ka_jV_j\right)\circ X_nd\mu_n \right|<\frac{\epsilon}{4}$ ;
		\item $\left\|\sum_{j=1}^ka_je_j^{(n)}-x_n \right\|<\frac{\epsilon}{4M\|h\|_0+1}$;
		\item for each $1\leq j\leq k$, $|a_j|\left|\sqrt{\frac{c_j^{(n)}}{c_j}}-1\right|\|e_j^{(n)}\|<\frac{\epsilon}{4kM\|h\|_0+1}$.
	\end{itemize}
	Then, we have, for any such $n>N$:
	
	\begin{align*}
		&\left|\int_{Q}gh\cdot\HU(x)d\HMU-\int_{\Omega_n}g_n\cdot\left(h\circ X_n\right)\cdot U_n(x_n)d\mu_n \right|\\
		&\leq\left|\int_{Q}gh\cdot\left(\HU(x)-\sum_{j=1}^ka_jV_j\right)d\HMU \right|+\left|\int_{Q}gh\sum_{j=1}^ka_jV_jd\HMU-\int_{\Omega_n}g_n\cdot\left(h\circ X_n\right)\cdot U_n(x_n)d\mu_n\right|\\
		&\leq \int_{Q} |gh|\left|\HU(x)-\sum_{j=1}^ka_jV_j \right| d\HMU+\left|\int_{Q}gh\sum_{j=1}^ka_jV_jd\HMU-\int_{\Omega_n}g_n\cdot\left(h\circ X_n\right)\cdot U_n(x_n)d\mu_n\right|\\
		&\leq \|h\|_0\|g\|_{\HMU}\left\|\HU\left(x-\sum_{j=1}^ka_je_j\right)\right\|_{\HMU}+\left|\int_{Q}gh\sum_{j=1}^ka_jV_jd\HMU-\int_{\Omega_n}g_n\cdot\left(h\circ X_n\right)\cdot U_n(x_n)d\mu_n\right|\\
		&<\frac{\epsilon}{4}+\left|\int_{Q}gh\sum_{j=1}^ka_jV_jd\HMU-\int_{\Omega_n}g_n\cdot\left(h\sum_{j=1}^ka_jV_j\right)\circ X_nd\mu_n\right|\\
		&\qquad\qquad+\left|\int_{\Omega_n}\left(g_n\cdot\left(h\sum_{j=1}^ka_jV_j\right)\circ X_n-g_n\cdot(h\circ X_n)\cdot U_n(x_n) \right)d\mu_n\right|\\
		&<\frac{2\epsilon}{4}+\left|\int_{\Omega_n}g_n\cdot(h\circ X_n)\left(\sum_{j=1}^k a_jV_j\circ X_n-U_n(x_n)\right)d\mu_n\right|\\
		&\leq\frac{\epsilon}{2}+\|h\|_0\|g_n\|_{\mu_n}\left\|\sum_{j=1}^k a_jV_j\circ X_n-U_n(x_n)\right\|_{\mu_n}\\
		&\leq\frac{\epsilon}{2}+M\|h\|_0\left(\left\|\sum_{j=1}^ka_j\left(V_j\circ X_n-U_n(e_j^{(n)})\right)\right\|_{\mu_n}+\left\|U_n\left(\sum_{j=1}^ka_je_j^{(n)}-x_n\right)\right\|_{\mu_n}\right)\\
		&=\frac{\epsilon}{2}+M\|h\|_0\left(\left\|\sum_{j=1}^ka_j\left(\sqrt{\frac{c_j^{(n)}}{c_j}}-1\right)U_n(e_j^{(n)})\right\|_{\mu_n}+\left\|\sum_{j=1}^ka_je_j^{(n)}-x_n \right\|\right)\\
		&<\frac{3\epsilon}{4}+M\|h\|_0\sum_{j=1}^k|a_j|\left|\sqrt{\frac{c_j^{(n)}}{c_j}}-1 \right|\|e_j^{(n)}\|\\
		&<\epsilon.
	\end{align*}
	
	Since $\epsilon$ is arbitrary, we have that \[\int_{Q}gh\cdot\HU(x)d\HMU=\lim_{n\rightarrow\infty}\int_{\Omega_n}g_n\cdot\left(h\circ X_n\right)\cdot U_n(x_n)d\mu_n,\] as claimed.
	
	Since $(g_n\cdot U_n(x_n))$ is bounded by hypothesis, and since we just established that $\Phi_n(g_n\cdot U_n(x_n))$ is weak* convergent, we have by definition that $g_n\cdot U_n(x_n)$ is $S$-convergent, and Proposition \ref{proposition_L2_limit} establishes unique $L_2(Q,\HMU)$ $S$-limit, which we can call $g'$. 
	
	Cauchy-Schwarz inequality shows that $g\cdot \HU(x)\in L_1(Q,\HMU).$ To show that $g\cdot \HU(x)\in L_2(Q,\HMU)$ and $g\cdot \HU(x)=\lim^{(S)}_{n\rightarrow\infty} g_n\cdot U_n(x_n)$, it is thus sufficient to show that $g\cdot \HU(x)=g'$ in $L_1(Q,\HMU).$
	
	We know that for any $h\in C(Q),$ $\int_{Q}hg'd\HMU=\int_{Q}hg\cdot\HU(x)d\HMU.$ Thus, if we define $\hat{\nu}$ as the $\mathbb{K}$-signed measure on $\Borel(Q)$ with $d\hat{\nu}=(g'-g\cdot\HU(x))d\HMU$, we know that $\int_{Q}h d\hat{\nu}=0$ for all $h\in C(Q).$ Thus, the total variation of $\hat{\nu}$ is $0$, and $g'-g\cdot\HU(x)=0$, concluding the proof.
\end{proof}

We have  all the required tools to prove the theorem. \ref{theorem_multiplication_operator}. 

\begin{proof}[Proof of Theorem \ref{theorem_multiplication_operator}]
	We want to show that there exists a unique real-valued $\Hm\in L_2(Q,\HMU)$ such that for any $x\in\Dom(A)$, $\HU(Ax)=\Hm\cdot\HU(x).$ We fully show existence first. 
	
	We consider our previously defined eigenvalue function $\lambda_n:\Omega_n\rightarrow\mathbb{R}$ for $n\in\mathbb{N}$. We know that for any $n\in\mathbb{N}$, using Property (\ref{property_compat_A}) of $A$-sampling-scale sequences to find $C\in\mathbb{R}$:
	\begin{align*}
		\int_{\Omega_n}|\lambda_n|^2d\mu_n&=\sum_{f\in\Omega_n}|\lambda_n(f)|^2\mu_n(f)\\
		&=\sum_{f\in\Omega_n}|\lambda_n(f)|^2\sum_{j=1}^{N_n}c_j^{(n)}|\LL e_j^{(n)},f\RR|^2\\
		&=\sum_{f\in\Omega_n}\sum_{j=1}^{N_n}c_j^{(n)}|\LL e_j^{(n)},\lambda_n(f)f\RR|\\
		&=\sum_{f\in\Omega_n}\sum_{j=1}^{N_n}c_j^{(n)}|\LL e_j^{(n)},A_nf\RR|\\
		&=\sum_{j=1}^{N_n}c_j^{(n)}\sum_{f\in\Omega_n}|\LL A_n e_j^{(n)}, f\RR|^2\\
		&=\sum_{j=1}^{N_n}c_j^{(n)}\|A_ne_j^{(n)}\|^2\leq C.
	\end{align*}
	
	We cannot quite use Proposition \ref{proposition_converging_isometry}, as we have not shown that $\Phi_n(\lambda_n)$ weak* converges. While it holds (see the following Remark \ref{remark_convergent_eigenvalue}), we will not use it here. Instead, we note that $\|\Phi_n(\lambda_n)\|^2\leq\int_{\Omega_n}|\lambda_n|^2d\mu_n\leq C$. By using Banach-Alaoglu again, there exist a weak* convergence subsequence $(\Phi_{n_m}(\lambda_{n_m}))_{n\in\mathbb{N}}$, which we fix.
	
	We consider the $A$-sampling scale subsequence $S'=(S_{n_m})_{n\in\mathbb{N}}$, which is also convergent, and for which $\int_{Q} hd\HMU=\lim_{m\rightarrow\infty}\int_{\Omega_{n_m}}h\circ X_{m_n}d\mu_{n_m}$ holds for any $h\in C(Q)$. We also note that since $\lim_{m\rightarrow\infty}e_j^{(n_m)}=e_j,$ the isometry $\HU$ defined by Proposition \ref{proposition_unique_isometry} is the same for the subsequence.  With respect to it, $(\lambda_{n_m})_{m\in\mathbb{N}}$ is now a $S'$-convergent sequence. Let $\hat{m}\in L_2(Q,\HMU)$ be its limit, which is real-valued by Proposition \ref{proposition_L2_limit}. 
	
	We show that for any $x\in\Dom(A)$, $\Hm\cdot\HU(x)=\HU(Ax).$ Let $(x_n\in H_n)_{n\in\mathbb{N}}$ be a sequence satisfying Property (\ref{property_sampling_approx}) of Definition \ref{definition_sampling}. We then note that in $L_2(\Omega_{n_m},\mu_{n_m}),$ $\|\lambda_{n_m}\cdot U_{n_m}(x_{n_m})\|=\|U_{m_n}(A_{n_m}x_{m_n})\|=\|A_{n_m}x_{n_m}\|,$ which converges and thus is bounded on $m$. Therefore, we can apply Proposition \ref{proposition_converging_isometry} on both $\Hm=\lim^{(S')}_{m\rightarrow\infty} \lambda_{n_m}$ and $\mathbf{1}_Q=\lim^{(S')}_{m\rightarrow\infty}\mathbf{1}_{\Omega_{n_m}},$ so that $\Hm\cdot \HU(x)\in L_2(Q,\HMU),$ and $$\Hm\cdot\HU(x)={\lim}^{(S')}_{m\rightarrow\infty} \lambda_{n_m}\cdot U_{n_m}(x_{n_m})={\lim}^{(S')}_{m\rightarrow\infty}U_{n_m}(A_{n_m}x_{n_m})=\HU(Ax). $$ This concludes the proof of existence of $\Hm$.
	
	Suppose $\Hm_1$ and $\Hm_2$ are two real-valued measurable functions satisfying the Theorem, and let $T$ be the multiplication operator on $L_2(Q,\HMU)$ induced by $\Hm_1-\Hm_2$. Then, for any $x\in\Dom(A),$  we have $T(\HU(x))=\HU(Ax)-\HU(Ax)=0$. Since $T$ is self-adjoint thus closed, since $\Dom(A)$ is dense in $H$ and since $\HU$ is an isometry, we have that for any $x\in H,$ $\HU(x)\in\Dom(T)$ and $T\HU(x)=0.$ In particular, for any $j\in\mathbb{N},$ since $\pi_j=\HU(\sqrt{c_j}e_j),$ $\pi_j\in\Dom(T)$ and $0=T\pi_j=(\Hm_1-\Hm_2)\cdot\pi_j.$  
	
	For $K=(\Hm_1-\Hm_2)^{-1}(0),$ and $C_j=\pi_j^{-1}(0),$ we thus have $\HMU(K\cup C_j)=1$ for any $j\in\mathbb{N}.$ Thus, $1=\HMU(\bigcap_{j\in\mathbb{N}} K\cup C_j)=\HMU(K\cup\bigcap_{j\in\mathbb{N}}C_j)=\HMU(K\cup\{O\}),$ where $O$ is defined as per Proposition \ref{proposition_null_origin}, and as such $\HMU(\{O\})=0$. Thus, $\HMU(K)=1,$ and $\Hm_1$ and $\Hm_2$ are equal $\HMU$-almost everywhere, and so share the same $L_2$ class.
\end{proof}
\begin{remark}\label{remark_convergent_eigenvalue}
	Since we showed that $(\Phi_n(\lambda_n))_{n\in\mathbb{N}}$ is bounded in $C(Q)',$ every subsequence $(\Phi_{n_m}(\lambda_{n_m}))_{m\in\mathbb{N}}$ has a weak* convergent subsequence $(\Phi_{n_{m_k}}(\lambda_{n_{m_k}}))_{k\in\mathbb{N}}$, which induces a $(S_{n_{m_k}})_{k\in\mathbb{N}}$-limit $\Hm'$, so that $\HU(Ax)=\Hm'\HU(x)$ holds for any $x\in\Dom(A)$ as shown in the proof. But by the uniqueness property shown, $\Hm'=\Hm$, and so the weak* limit of $(\Phi_{n_{m_k}}(\lambda_{n_{m_k}}))_{k\in\mathbb{N}}$ is given by $h\rightarrow\int_{Q}h\Hm d\HMU$ on $h\in C(Q).$ Thus, the whole sequence $(\Phi_n(\lambda_n))_{n\in\mathbb{N}}$ must weak* converge to the same limit, and so $(\lambda_n)_{n\in\mathbb{N}}$ is $S$-convergent with $\lim_{n\rightarrow\infty}^{(S)} \lambda_n = \Hm.$ It should possible to show directly, with the idea that $A_n e_j^{(n)}$ weakly converges in $H$, and that for any continuous $h$ on $Q$, $U_n^{-1}((h\circ X_n)\cdot U_n (x_n))$ norm-converges to $U^*(h\cdot \HU(x))$. 
\end{remark}

\section{Example: Shift Operator}\label{section_shift}

In the next sections, we study specific operators, with the intentions of using the process to extract explicit descriptions of resulting spaces. We start with the shift operator, because of its richness despite its simplicity. In this case, the natural spectral embedding is already  widely known to be given by the Fourier series on $L_2([0,1])$. Here, we find that our process precisely, and naturally outputs this embedding, up to measure-preserving mapping.

Let $\mathbb{K}=\mathbb{C},$ $H=\ell_2(\mathbb{Z})$ with canonical Hilbert basis $(g_l)_{l\in\mathbb{Z}}$. Then, we consider $A=\frac{1}{2}(RS+LS)$, where $LS$ and $RS$ are the left-shift and right-shift operators on $\ell_2(\mathbb{Z}),$ so that $RSg_{l-1}=g_{l}=LSg_{l+1}$. We have that since $RS$ and $LS$  are unitary on $H$ and inverses of each other, $A$ is a bounded self-adjoint operator on $H$, with $\|A\|\leq 1$. We also define the sequence in $\mathbb{Z}$ $(l_j)_{j\in\mathbb{N}}$ given by $(l_1,l_2,l_3,l_4,\dots)=(0,1,-1,2,\dots).$


We first construct what we consider to be a suitable $A$-sampling sequence. When working with specific operators, such a sequence should allow for explicit description of the resulting objects. As such, it should use the relevant specific properties of the operator.

First, we define all relevant objects given $n\in\mathbb{N}$. Let $H_n=\Span(\{g_l\}_{l=-n}^n)$. Then, let $A_n=\frac{1}{2}(RS_n+LS_n),$ where $LS_n$ and $RS_n$ are respectively the left-shift and the right-shift operators modulo $2n+1$. Specifically, they are defined by linearity and $$RS_n g_l=\begin{cases}
	g_{-n} \quad \text{ if $l=n$ }\\
	g_{l+1} \quad \text{ otherwise }
\end{cases} \quad \text{ and }\quad LS_ng_l=\begin{cases}
g_n \quad \text{ if $l=-n$ }\\
g_{l-1}\quad \text{ otherwise. }
\end{cases}$$

Then, let $\Omega_n=\{f_k^{(n)}\}_{k=0}^{2n},$ where $$f_k^{(n)}=\frac{1}{\sqrt{2n+1}}\sum_{l=-n}^n e^{-2\pi i l \frac{k}{2n+1}}g_l.$$

Furthermore, let $N_n=2n+1$, $e_j^{(n)}=g_{l_j},$ and $c_j^{(n)}=\frac{1}{2^j(1-2^{N_n})}$ for $1\leq j\leq 2n+1.$ Finally, let $S_n=(H_n,A_n,\Omega_n, (e_j^{(n)},c_j^{(n)})_{j=1}^{N_n}).$ We then show the following.

\begin{prop}\label{proposition_shift_sampling_sequence}
	We have that $(S_n)_{n\in\mathbb{N}}$ forms a strong $A$-sampling-scale sequence. Furthermore, $\lambda_n(f_k^{(n)})=\cos\left(2\pi\frac{k}{2n+1} \right),$ $e_j=g_{l_j}$ and $c_j=\frac{1}{2^j}$ for any $n, j\in\mathbb{N}.$
\end{prop}
\begin{proof}
	Many parts of this are mostly calculations or quickly established. We go through the list of properties of Definitions \ref{definition_sampling}, \ref{definition_scale} and \ref{definition_sampling_scale_sequence}.
	
	\textbf{Property (\ref{property_sampling_finite_dim}):} As the span of a finite subset of $H$, $H_n$ is a finite-dimensional subspace of $H$ by definition. Of note, $(g_l)_{l=-n}^n$ is an orthonormal basis of $H_n$.
	
	\textbf{Property (\ref{property_sampling_symmetric_op}):} As permutations on an orthonormal basis of $H_n$ extended by linearity, $LS_n$ and $RS_n$ are both unitary maps on $H_n$, and they are inverse, thus adjoint, of each other. Therefore, $A_n$ must be a linear symmetric operator on $H_n.$ We also note as a result that $\|A_n\|\leq 1.$
	
	\textbf{Property (\ref{property_sampling_eigenbasis}):} We first establish that $\Omega_n$ is an orthonormal basis of $H_n$.For $0\leq k,m\leq 2n$, we have  $\LL f_k^{(n)}, f_m^{(n)} \RR=\frac{1}{2n+1}\sum_{l=-n}^n e^{2\pi i \frac{m-k}{2n+1}l},$ which $1$ if $m=k,$ or, through geometric summation, $0$ if $m\neq k.$ Thus, $\Omega_n\subset H_n$ is orthonormal, and since $|\Omega_n|=2n+1=\dim(H_n)$, it is an orthonormal basis of $H_n$. 
	
	To simplify the next calculation, let $\phi$ be the permutation on $\{l\}_{l=-n}^n$ so that $RS_n g_l=g_{\phi(l)}$ holds on $g_l\in H_n$, noting $\phi(l)\equiv l+1$ modulo $2n+1$. We then calculate, for any $f_k^{(n)}\in\Omega_n,$
	\begin{align*}
		RS_n f_k^{(n)}&=\frac{1}{\sqrt{2n+1}}\sum_{l=-n}^n e^{-2\pi i l \frac{k}{2n+1}}g_{\phi(l)}=\frac{e^{2\pi i\frac{k}{2n+1}}}{\sqrt{2n+1}}\sum_{l=-n}^n e^{-2\pi i (l+1) \frac{k}{2n+1}}g_{\phi(l)}\\
		&=\frac{e^{2\pi i\frac{k}{2n+1}}}{\sqrt{2n+1}}\sum_{l=-n}^n e^{-2\pi i \phi(l) \frac{k}{2n+1}}g_{\phi(l)}=\frac{e^{2\pi i\frac{k}{2n+1}}}{\sqrt{2n+1}}\sum_{l=-n}^n e^{-2\pi i l \frac{k}{2n+1}}g_{l}\\
		&=e^{2\pi i\frac{k}{2n+1}}f_k^{(n)}.
	\end{align*}
	The same way, or using that $LS_n$ and $RS_n$ are inverses of each other, we find $LS_n f_k^{(n)}=e^{-2\pi i\frac{k}{2n+1}}f_k^{(n)}.$ Thus, we find that $A_n f_k^{(n)}= \cos\left(2\pi\frac{k}{2n+1}\right).$ We conclude that $\Omega_n$ is an orthonormal eigenbsis of $A_n$, with $\lambda_n(f_k^{(n)})=\cos\left(2\pi\frac{k}{2n+1}\right).$
	
	\textbf{Property (\ref{property_sampling_approx}):} Let $x\in H.$ For $n\in\mathbb{N},$  let $x_n=\Proj_{H_n}x=\sum_{l=-n}^n \LL x, g_l\RR g_l.$ We establish that $x=\lim_{n\rightarrow\infty} x_n$ and $Ax=\lim_{n\rightarrow\infty} A_n x_n$, which is sufficient to show the property.
	
	The first limit is simply a consequence of Parseval's Identity. Furthermore, since $A$ is bounded on $H$, we have $Ax=\lim_{n\rightarrow\infty} Ax_n.$ Thus, we conclude by showing $\lim_{n\rightarrow\infty} Ax_n-A_nx_n=0$. Whenever $l\neq\pm n,$ we have $A g_l= A_n g_l.$ 
	
	Thus, we have $$\|Ax_n-A_nx_n\|=\|(A-A_n)(\LL x, g_n\RR g_n+\LL x, g_{-n}\RR g_{-n})\|\leq 2(|\LL x, g_n\RR|+|\LL x, g_{-n}\RR|).$$ 
	
	By Bessel's inequality, both $\LL x, g_n\RR$ and $\LL x, g_{-n}\RR$ converge to $0$, and so we conclude $\lim_{n\rightarrow\infty}\|Ax_n-A_nx_n\|=0.$
	
	\textbf{Property (\ref{property_scale_elements})} is trivial.
	
	\textbf{Property (\ref{property_scale_proba}):} We have $\sum_{j=1}^{N_n}c_j^{(n)}\|e_j^{(n)}\|^2=\frac{1}{1-2^{N_n}}\sum_{j=1}^{N_n}\frac{1}{2^j}=1.$
	
	\textbf{Property (\ref{property_scale_infinity})} is trivial.
	
	\textbf{Property (\ref{property_scale_limits}):} For $j\in \mathbb{N},$  we have $e_j^{(n)}=g_{l_j}$ and $c_j^{(n)}=\frac{1}{2^j(1-2^{2n+1})}$ whenever $N_n\geq j$. Thus, $\lim_{n\rightarrow\infty}e_j^{(n)}=g_{l_j}\neq 0$, and $\lim_{n\rightarrow\infty}c_j^{(n)}=\frac{1}{2^j}>0.$ We conclude that the property holds, and that $e_j=g_{l_j}$ and $c_j=\frac{1}{2^j}.$  
	
	\textbf{Property (\ref{property_scale_dense})} follows from $\{e_j\}_{j\in\mathbb{N}}=\{ g_l\}_{l\in \mathbb{Z}}.$
	
	\textbf{Property (\ref{property_scale_bias}):} We have $\sum_{j\in\mathbb{N}} c_j\|e_j\|^2=\sum_{j\in\mathbb{N}}\frac{1}{2^j}=1.$
	
	\textbf{Property (\ref{property_compat_H}) and (\ref{property_compat_Omega}):} We have  $\{l_j\}_{j=1}^{2n+1}=\{l\}_{l=-n}^n$, thus the set  $\{e_j^{(n)}\}_{j=1}^{N_n}=\{g_l\}_{l=-n}^n\subset H_n$ is an orthonormal basis of $H_n.$
	
	\textbf{Property (\ref{property_compat_A}):} Since we established earlier that $\|A_n\|\leq 1$, we have $$\sum_{j=1}^{N_n}c_j^{(n)}\|A_ne_j^{(n)}\|^2\leq\sum_{j=1}^{N_n}c_j^{(n)}\|e_j^{(n)}\|^2=1.$$
	
	\textbf{Propoerty (\ref{property_compat_strong}):} We know that whenever $|l|<n$, we have $g_l\in H_n$ and $A g_l=A_n g_l.$ In other words, for any $j\in\mathbb{N},$ we have that $A_n e_j^{(n)}=A e_j$ whenever  $n>|l_j|$ and $N_n\geq j$,, and so $\lim_{n\rightarrow\infty} A_ne_j^{(n)}=Ae_j\in H.$
\end{proof}

With this sampling-scale sequence, we will first state the final result.

\begin{definition}
	Let $\mu$ be the Lebesgue measure on $\Borel([0,1]).$ Furthermore, let $X:[0,1]\rightarrow Q$ be given by $$X(t)=\left(\frac{1}{\sqrt{2}^j} (e^{2\pi i t})^{l_j}\right)_{j\in\mathbb{N}}.$$
\end{definition}
\begin{remark}
	Since $l_j\in\mathbb{Z},$ it is clear that $\pi_j\circ X$ is continuous for any $j\in\mathbb{N},$ and so $X$ is continuous.
\end{remark}
\begin{theorem}\label{theorem_shift_fourier_series}
	The sequence $(S_n)_{n\in\mathbb{N}}$ is convergent, with the resulting $\HMU$ being given by the pushforward measure $\HMU=X_{\#}(\mu).$ Furthermore, if $U:H\rightarrow L_2([0,1],\mu)$ is given by $U(z)=\HU(z)\circ X,$ then $U$ is an isometry for which $(U(g_l))(t)=e^{2\pi i l t}$ for almost all $t\in [0,1].$ Finally, if $m:[0,1]\rightarrow\mathbb{R}$ is given by $m=\Hm\circ X,$ then $m(t)=\cos(2\pi t)$ almost everywhere in $[0,1]$, and for any $z\in H$, $$(U(Az))=m\cdot U(z).$$
\end{theorem}

Thus, given that sampling-scale sequence, the resulting isometry associates a sequence with its Fourier series, and the resulting multiplication operator is the cosine function of unit period, up to a measure-preserving map.

We will not provide a detailed proof here, to avoid needless extra technical details. Instead, we will give the step-by-step idea, and the reader may complete the details at their convenience. 

A complete proof of Theorem \ref{theorem_shift_fourier_series} is detailed in Appendix \ref{appendix_shift}. However it uses the results of Appendix \ref{appendix_nonstandad}, which fundamentally uses nonstandard analysis. The reader will need familiarity with the domain and with the notions of \cite{nonez2024spectralequivalencesnonstandardsamplings} to read that part.

\begin{proof}[Key Steps for the Proof of Theorem \ref{theorem_shift_fourier_series}]
	In order, we can establish the following:
	\begin{itemize}
		\item First, we can calculate directly that for any $f_k^{(n)}\in \Omega_n$ and $j\leq N_n,$ $\LL e_j^{(n)}, f_k^{(n)} \RR= \frac{1}{\sqrt{2n+1}}e^{2\pi i l_j\frac{k}{2n+1}}.$ As such $\mu_n(V)=\frac{|V|}{2n+1}$ results as the uniform measure  Furthermore, for any $f_k^{(n)}\in\Omega_n$ and $j\leq N_n,$ we calculate that $(U_n(e_j^{(n)}))(f_k^{(n)})=e^{2\pi i l_j \frac{k}{2n+1}}$.
		\item We can directly establish that if $(f_{k_n}^{(n)})\in\mathbb{N}$ is any integer sequence for which $\lim_{n\rightarrow\infty} \frac{k_n}{2n+1}=t\in[0,1]$, then $\lim_{n\rightarrow\infty} X_n(f_{k_n}^{(n)})= X(t).$
		\item We have that for any $h\in C(Q),$ $$\int_{\Omega_n}(h\circ X_n)d\mu_n=\int_{[0,1]}\left(\sum_{k=0}^{2n}h(X_n(f_k^{(n)}))\mathbf{1}_{[\frac{k}{2n+1},\frac{k+1}{2n+1})}\right)d\mu.$$ Since the sequence of functions in that last integral is bounded and pointwise convergent to $h\circ X$ on $[0,1),$ we conclude from the Dominated Convergence Theorem that $(S_n)_{n\in\mathbb{N}}$ is convergent, with $\HMU=X_{\#}(\mu),$ the pushforward measure.
		\item We can then calculate that for any $j\in\mathbb{N}$ and $t\in [0,1]$, $(V_j\circ X)(t)=e^{2\pi i l_j\cdot t},$ and so for any $l\in\mathbb{Z}$, $(U(g_l))(t)=e^{2\pi i l t}.$ That $U$ is isometric is a direct consequence from $X$ being measure invariant.
		\item We calculate that $\lambda_n=\frac{1}{2}\left(U_n(e_2^{(n)})+U_n(e_3^{(n)})\right).$ Thus, it follows from Proposition \ref{proposition_converging_isometry} and Remark \ref{remark_convergent_eigenvalue} that $\Hm= \frac{1}{2}\left(V_2+V_3\right),$ and so we have $m(t)=(\Hm\circ X)(t)=\frac{1}{2}\left(e^{2\pi i l_2 t}+e^{2\pi i l_3 t}\right)=\cos(2\pi t)=$ for any $t\in [0,1].$
		\item Finally, that $U(Az)=m\cdot U(z)$ always hold directly follows from the definitions of $U$ and $m$ and the previous fact.
	\end{itemize}
\end{proof}

\section{Example: Differential operator on $\mathbb{R}$}\label{section_differential} Here, we consider another example, that of the derivative for compactly supported smooth functions on $\mathbb{R}$. We think it is maybe the most important one, as it opens the door to further research on some classes of differential operators. Just like with the shift operator, there is a natural known spectral embedding, which is given by the Fourier Transform. Here, we find that up to measure-preserving map and normalization, our process naturally outputs this embedding.

Let $\mathbb{K}=\mathbb{C}$ and $H=L_2(\mathbb{R},\mu),$ with $\mu$ being the Lebesgues measure on $\mathbb{R}.$ Furthermore, let $\Dom(A)=C_c^{\infty}(\mathbb{R}),$ and let $A=-i\frac{d}{dx}.$ It is known that $C_c^{\infty}(\mathbb{R})$ is dense in $H$, and it is quickly verified though integration by parts that $A$ is symmetric.

\subsection{Finding the sampling sequence}\label{subsection_differential_sampling} Our intuition is that to expect the most explicit and natural resulting objects, one should be meticulous in choosing  the sampling most "tailored" to the given operator. In the case of the shift operator, we only had to make minor modifications on $A_n$ when compared with the generic case described in Proposition \ref{proposition_sampling_existence}, so that $\Omega_n$ could have an explicit description. Here, we think that the better strategy is to construct the sampling with inspiration from numerical methods of the derivative.

Given $n\in\mathbb{N},$ and $l\in\mathbb{Z}$ such that $-n!^2\leq l <n!^2,$ let $s_l^{(n)}=[\frac{l}{n!},\frac{l+1}{n!}).$  Then, let $H_n=\Span(\{\mathbf{1}_{s_l^{(n)}}\}_{l=-n!^2}^{n!^2-1}).$ we note that $(\mathbf{1}_{s_{l}^{(n)}})_{l=-n!^2}^{n!^2-1}$ forms on orthogonal basis of $H_n$, all having the same norm $\|\mathbf{1}_{s_{l}^{(n)}}\|=\frac{1}{\sqrt{n!}}.$ Thus, we can define $LS_n$ and $RS_n$ on $H_n$ to be the left and right shift operators mod $2n!$ with respect to that basis. Specifically, like in the previous section, they are extending by linearity the relations $$RS_n \mathbf{1}_{s_{l}^{(n)}} =\begin{cases}
	 \mathbf{1}_{s_{-n!^2}^{(n)}} \quad \text{ if $l=n!^2-1$ }\\
	 \mathbf{1}_{s_{l+1}^{(n)}} \quad \text{ otherwise }
\end{cases} \quad \text{ and }\quad LS_n\mathbf{1}_{s_{l}^{(n)}} =\begin{cases}
	\mathbf{1}_{s_{n!^2-1}^{(n)}} \quad \text{ if $l=-n!^2$ }\\
	\mathbf{1}_{s_{l-1}^{(n)}} \quad \text{ otherwise. }
\end{cases}$$

Then, we define $A_n=-i\frac{LS_n-RS_n}{2/n!}$ on $H_n$. Finally, let $\Omega_n=\{f_k^{(n)}\}_{k=-n!^2}^{n!^2-1}$ with $$f_k^{(n)}=\frac{1}{\sqrt{2n!}}\sum_{l=-n!^2}^{n!^2-1}e^{2\pi i l\frac{k}{2n!^2}}\mathbf{1}_{s_l^{(n)}}.$$

We first prove the following:
\begin{prop}
	We have that $(H_n,A_n,\Omega_n)_{n\in\mathbb{N}}$ forms a sampling sequence for $A$. Furthermore, $\lambda_n(f_k^{(n)})=n!\sin\left(\pi\frac{k}{n!^2}\right).$ Finally, $LS_n$ and $RS_n$ are unitary and adjoint of each other, with $LS_n f_k^{(n)}=e^{2\pi i \frac{k}{2n!^2}}f_k^{(n)}$ for any $f_k^{(n)}\in\Omega_n.$
\end{prop} 
\begin{proof}
		\textbf{Property (\ref{property_sampling_finite_dim}):} As the span of a finite subset of $H$, $H_n$ is a finite-dimensional subspace of $H$ by definition. Of note, we have earlier established that $(\sqrt{n!}\mathbf{1}_{s_{l}^{(n)}})_{l=-n!^2}^{n!^2-1}$ is an orthonormal basis of $H_n$, and so $\dim(H_n)=2n!$.
	
	\textbf{Property (\ref{property_sampling_symmetric_op}):} Since $LS_n$ and $RS_n$ both permute orthonormal basis of $H_n$ given by  $(\sqrt{n!}\mathbf{1}_{s_{l}^{(n)}})_{l=-n!^2}^{n!^2-1}$, they both are unitary. Furthermore, they are inverses of each other, thus adjoint of each other. Thus, $LS_n-RS_n$ is antisymmetric, making $A_n$ symmetric.
	
	\textbf{Property (\ref{property_sampling_eigenbasis}):} We first establish that $\Omega_n$ is an orthonormal basis of $H_n$. We have, for $-n!^2\leq k <n!^2,$ that $\|f_k^{(n)}\|^2=\frac{1}{2n!}\sum_{l=-n!^2}^{n!^2-1}\frac{1}{n!}=1.$ Furthermore, for  $-n!^2\leq k_2<k_1 <n!^2,$ we have, though geometric sums, that $$2n!^2\LL f_{k_1}^{(n)},f_{k_2}^{(n)}\RR=\sum_{l=-n!^2}^{n!^2-1}\left(e^{2\pi i\frac{k_1-k_2}{2n!^2}}\right)^l=0. $$ Thus, $\Omega_n$ is orthonormal. Since $|\Omega_n|=2n!^2=\dim(H_n),$ we have that $\Omega_n$ is indeed an orthonormal basis of $H_n$.
	
	Just like with last section, we simplify the next calculation by considering the permutation  $\phi$ on $\{l\}_{l=-n!^2}^{n!^2-1}$ for which $LS_n \mathbf{1}_{s_l^{(n)}}=\mathbf{1}_{s_{\phi(l)}^{(n)}}$, noting $\phi(l)\equiv l-1$ modulo $2n!^2$. We then calculate, for $-n!^2\leq k <n!^2,$
	\begin{align*}
		LS_n f_k^{(n)}&=\frac{1}{\sqrt{2n!}}\sum_{l=-n!^2}^{n!^2-1}e^{2\pi i l\frac{k}{2n!^2}}\mathbf{1}_{s_{\phi(l)}^{(n)}}\\
		&=e^{2\pi i\frac{k}{2n!^2}}\frac{1}{\sqrt{2n!}}\sum_{l=-n!^2}^{n!^2-1}e^{2\pi i \phi(l)\frac{k}{2n!^2}}\mathbf{1}_{s_{\phi(l)}^{(n)}}\\
		&=e^{2\pi i\frac{k}{2n!^2}}\frac{1}{\sqrt{2n!}}\sum_{l=-n!^2}^{n!^2-1}e^{2\pi i l\frac{k}{2n!^2}}\mathbf{1}_{s_l^{(n)}}\\
		&=e^{2\pi i\frac{k}{2n!^2}} f_k^{(n)}.
	\end{align*}
	
	Using that $RS_n$ and $LS_n$ are inverses of each other, we find $RS_n=e^{-2\pi i\frac{k}{2n!^2}} f_k^{(n)}.$ Therefore, $A_nf_k^{(n)}=n!\frac{e^{2\pi i\frac{k}{2n!^2}}-e^{-2\pi i\frac{k}{2n!^2}}}{2i}f_k^{(n)}=n!\sin\left(\pi\frac{k}{n!^2}\right)f_k^{(n)}.$ We thus note that $\Omega_n$ is an orthonormal eigenbasis of $A_n$, with $\lambda_n(f_k^{(n)})=n!\sin\left(\pi\frac{k}{n!^2}\right).$
	
	\textbf{Property (\ref{property_sampling_approx}):} Let $u\in \Dom(A)=C_c^{\infty}(\mathbb{R}).$ Furthermore, given $n\in\mathbb{N}$, let $$u_n=\sum_{l=-n!^2}^{n!^2-1} u(\frac{l}{n!})\mathbf{1}_{s_l^{(n)}}\in H_n.$$ We establish that in $L_2(\mathbb{R},\mu),$ $u=\lim_{n\rightarrow\infty}u_n$ and $Au=\lim_{n\rightarrow\infty} A_nu_n,$ which is sufficient to conclude the property as $u$ is arbitrary.
	
	First, let $r\in\mathbb{R}$ for which $r\geq 1$ and $\operatorname{supp}(u)\subset [-r+1,r-1].$ We first establish that all $u_n$ are supported in $[-r,r].$ We have that for $t\in[-n!,n![\supset \operatorname{supp}(u_n),$ $u_n(t)=u(\frac{l_n(t)}{n!}),$ where $-n!^2\leq l_n(t)< n!^2$ is the unique such integer for which $t\in s_{l_n(t)}^{(n)}.$ Noting that $\operatorname{diam}(s_l^{(n)})=\frac{1}{n!}\leq 1$ , for any $t\notin [-r,r]$ but still inside $[-n!,n![,$ $l_n(t)\notin [-r+1,r-1],$ and so $u_n(t)=0.$ It follows that $u$ and all $u_n$ are supported in $[-r,r]$. 
	
	We set $M\in\mathbb{R}$ such that $|u'(t)|\leq M$ for all $t\in[-r,r],$ which exists since $u$ is smooth.  Let $n$ large enough that $[-r,r]\subset [-n!, n![.$ For any such $n$, we have that for any $t\in[-r,r]$, $|u(t)-u_n(t)|=|u(t)-u(l_n(t))|\leq M|t-l_n(t)|\leq\frac{M}{n!}.$ Thus, $$\|u-u_n\|^2=\int_{\mathbb{R}}|u-u_n|^2d\mu=\int_{[-r,r]}|u-u_n|^2d\mu\leq 2r\frac{M^2}{n!^2}.$$ We conclude that $(u_n)_{n\in\mathbb{N}}$ converges to $u$.
	
	We have that for any $n\in\mathbb{N}$ for which $n!>r,$ $u(t)=0$ for any $t\notin(-n!,n!-\frac{1}{n!}).$ Thus, we have that $u_n=\sum_{l=-n!^2+1}^{n!^2-2} u(\frac{l}{n!})\mathbf{1}_{s_l^{(n)}}$ and thus $$LS_nu_n=\sum_{l=-n!^2+1}^{n!^2-2} u(\frac{l}{n!})\mathbf{1}_{s_{l-1}^{(n)}}=\sum_{l=-n!^2}^{n!^2-3}u(\frac{l+1}{n!})\mathbf{1}_{s_l^{(n)}}=\sum_{l=-n!^2}^{n!^2-1}u(\frac{l+1}{n!})\mathbf{1}_{s_l^{(n)}}.$$ The same way, $$RS_nu_n=\sum_{l=-n!^2+1}^{n!^2-2} u(\frac{l}{n!})\mathbf{1}_{s_{l+1}^{(n)}}=\sum_{l=-n!^2+2}^{n!^2-1}u(\frac{l-1}{n!})\mathbf{1}_{s_l^{(n)}}=\sum_{l=-n!^2}^{n!^2-1}u(\frac{l-1}{n!})\mathbf{1}_{s_l^{(n)}}.$$
	
	Therefore, we calculate, for such $n$,
	\begin{align*}
		A_nu_n &=-i\frac{1}{2/n!}(LS_n u_n-RS_n u_n)\\
		&=-i\sum_{l=-n!^2}^{n!^2-1}\frac{u(\frac{l+1}{n!})-u(\frac{l-1}{n!})}{2/n!}\mathbf{1}_{s_l^{(n)}}\\
		&=-i\sum_{l=-n!^2}^{n!^2-1} u'(c_l^{(n)})\mathbf{1}_{s_l^{(n)}},
	\end{align*}
	where the existence of $c_l^{(n)}\in(\frac{l-1}{n!},\frac{l+1}{n!})$ is guaranteed by the Mean Value Theorem. Since $u$ is supported in $[-r+1,r-1]$, $\operatorname{supp}(u')\subset [-r+1,r-1]$ as well. Similarly as before, for any $t\in[n!,n![\setminus [-r,r],$ we have $c_{l_n(t)}^{(n)}\notin[-r+1,r-1]$ and so $(A_nu_n)(t)=-iu'(c_{l_n(t)}^{(n)})=0,$  since $|t-c_{l_n(t)}^{(n)}|<\frac{2}{n!}\leq 1$. Thus, for any $n$ such that $n!> r\geq 1$, $A_nu_n$ is supported in $[-r,r].$
	
	If $M_2\in\mathbb{R}$ is such that $|u''(t)|\leq M_2$ for all $t\in[-r,r],$ then we have that $$|(A_nu_n)(t)-(Au)(t)|=|(-iu'(c_{l_n(t)}^{(n)}))- (-iu'(t))|=|u'(t)-u'(c_{l_n(t)}^{(n)})|\leq M_2\frac{2}{n!}.$$ Thus, $$\|A_nu_n-Au\|^2=\int_{[-r,r]}|A_nu_n-Au|^2d\mu\leq 4M_2r\frac{1}{n!}.$$ Therefore, $(A_n u_n)_{n\in\mathbb{N}}$ converges to $Au$, ending the proof.
\end{proof}

\subsection{Finding a suitable scale}\label{subsection_differential_scale} While we could use Proposition \ref{proposition_scale_existence} to consider a generic scale, we meticulously build one here that helps simplify much of the calculations.

First, let $E:\mathbb{R}\rightarrow\mathbb{R}_{>0}$ with $E(t)=\left(\frac{2}{\pi}\right)^{\frac{1}{4}}e^{-t^2}.$ We have that $E\in H$ is even, with $\|E\|=1.$ We also note that $E$ is smooth, with all derivatives bounded and in $H$. 

Furthermore, let $(L_n)_{n\in\mathbb{N}}$ be an integer sequence such that $0<L_n<n!^2-1$ and   $\lim_{n\rightarrow\infty}\frac{n!^{\frac{3}{2}}}{L_n}=\lim_{n\rightarrow\infty}\frac{L_n}{n!^2}=0.$ For example, $L_n=\lfloor n!^{\frac{5}{3}}\rfloor$ works. Then, let $$E^{(n)}:=\sum_{l=-L_n}^{L_n}E\left(\frac{l}{n!}\right)\mathbf{1}_{s_l^{(n)}}+\frac{i}{n!}\mathbf{1}_{s_0^{(n)}}\in H_n\setminus\{0\}.$$

\begin{prop}\label{proposition_gaussian_finite}
	We have that $E=\lim_{n\rightarrow\infty}E^{(n)}$ and  $\lim_{n\rightarrow\infty}A_n E^{(n)}=-iE'.$ Furthermore, for any $n\in\mathbb{N}$ and $f\in\Omega_n,$ $\LL E^{(n)}, f\RR\neq 0.$ 
\end{prop}
\begin{proof}
	We first show that $E=\lim_{n\rightarrow\infty} E^{(n)}.$ We calculate, for $n\in\mathbb{N}$,
	\begin{align*}
		\|E-E^{(n)}-\frac{i}{n!}\mathbf{1}_{s_0^{(n)}}\|^2=\|E\mathbf{1}_{\mathbb{R}\setminus[-\frac{L_n}{n!},\frac{L_n+1}{n!}[}\|^2+\sum_{l=-L_n}^{L_n}\left\|\left(E-E\left(\frac{l}{n!}\right)\right)\mathbf{1}_{s_l^{(n)}}\right\|^2.
	\end{align*}
	
	Since both $-\frac{i}{n!}\mathbf{1}_{s_0^{(n)}}$ and $E\mathbf{1}_{\mathbb{R}\setminus[-\frac{L_n}{n!},\frac{L_n+1}{n!}[}$ converge to $0$, using Monotonic Convergence Theorem and the fact that $\frac{L_n}{n!}$ goes to infinity, we have that it is sufficient to establish $\lim_{n\rightarrow\infty}\sum_{l=-L_n}^{L_n}\left\|\left(E-E\left(\frac{l}{n!}\right)\right)\mathbf{1}_{s_l^{(n)}}\right\|^2=0$ to prove $E=\lim_{n\rightarrow\infty} E^{(n)}.$ Since $E'$ is bounded, let $M\in\mathbb{R}$ such that $|E'(t)|\leq M$ holds for $t\in\mathbb{R}.$ Then, by Mean Value Theorem, we have that whenever $|l|\leq L_n$ and $t\in s_l^{(n)}, $ $|E(t)-E(\frac{l}{n!})|\leq \frac{M}{n!}.$ Therefore, for such $l$, $\|\left(E-E\left(\frac{l}{n!}\right)\right)\mathbf{1}_{s_l^{(n)}} \|^2\leq \frac{M^2}{n!^3},$ and so $$\sum_{l=-L_n}^{L_n}\left\|\left(E-E\left(\frac{l}{n!}\right)\right)\mathbf{1}_{s_l^{(n)}}\right\|^2\leq\frac{(2L_n+1)M^2}{n!^3}. $$ Since $\lim_{n\rightarrow\infty}\frac{L_n}{n!^2}=0,$ we get $\lim_{n\rightarrow\infty}\sum_{l=-L_n}^{L_n}\left\|\left(E-E\left(\frac{l}{n!}\right)\right)\mathbf{1}_{s_l^{(n)}}\right\|^2=0,$ and $E=\lim_{n\rightarrow\infty}E^{(n)}.$
	
	Next, since $-n!^2<-L_n<L_n<n!^2-1$, we have $A_n\mathbf{1}_{s_l^{(n)}}=-i\frac{n!}{2}(\mathbf{1}_{s_{l-1}^{(n)}}-\mathbf{1}_{s_{l+1}^{(n)}})$ whenever $-L_n\leq l\leq L_n.$ Thus, 
	
	\begin{samepage}
	\begin{align*}
		A_n\left(E^{(n)}-\frac{i}{n!}\mathbf{1}_{s_0^{(n)}}\right)&=\sum_{l=-L_n}^{L_n}E\left(\frac{l}{n!}\right)A_n\mathbf{1}_{s_l^{(n)}}\\
		&=\frac{-in!}{2}\sum_{l=-L_n}^{L_n}E\left(\frac{l}{n!}\right)\left(\mathbf{1}_{s_{l-1}^{(n)}}-\mathbf{1}_{s_{l+1}^{(n)}}\right) \\
		&=\frac{-in!}{2}\left(\sum_{l=-L_n}^{L_n}\left(E\left(\frac{l+1}{n!}\right)-E\left(\frac{l-1}{n!}\right)\right)\mathbf{1}_{s_l^{(n)}}+R_n\right),
	\end{align*}
	\end{samepage}
	where $$R_n=E(\frac{-L_n}{n!})\mathbf{1}_{s_{-L_n-1}^{(n)}}-E(\frac{L_n+1}{n!})\mathbf{1}_{s_{L_n}^{(n)}}+E(\frac{L_n}{n!})\mathbf{1}_{s_{L_n+1}^{(n)}}-E(\frac{-L_n-1}{n!})\mathbf{1}_{s_{L_n}^{(n)}}.$$ We have that $$\|n!R_n\|^2=n!\left(|E(\frac{-L_n}{n!})|^2+|E(\frac{L_n+1}{n!})|^2+|E(\frac{L_n}{n!})|^2+|E(\frac{-L_n-1}{n!})|^2\right),$$ and so $\|n!R_n\|^2\leq 4n!\left(\frac{2}{\pi}\right)^{\frac{1}{4}}e^{-2(\frac{L_n}{n!})^2}$. Since $\frac{L_n}{n!^{\frac{3}{2}}}\rightarrow\infty$, for large enough $n$ we have $\frac{L_n}{n!}>\sqrt{n!},$ and so $\|n!R_n\|^2\leq 4\left(\frac{2}{\pi}\right)^{\frac{1}{4}}n!e^{-2 n!},$ which implies $\lim_{n\rightarrow\infty} n!R_n=0.$ We can also calculate that $\lim_{n\rightarrow\infty}A_n(\frac{i}{n!}\mathbf{1}_{s_0^{(n)}})=\lim_{n\rightarrow\infty}\frac{1}{2}(\mathbf{1}_{s_{-1}^{(n)}}-\mathbf{1}_{s_1^{(n)}})=0.$ 
	
	Therefore, if $F_n:=iA_n\left(E^{(n)}-\frac{i}{n!}\mathbf{1}_{s_0^{(n)}}\right)-\frac{n!}{2}R_n,$ then $\lim_{n\rightarrow\infty} F_n= E'$ if and only if $\lim_{n\rightarrow\infty}A_nE^{(n)}=-iE'.$
	
	Then, by the Mean Value Theorem, we have elements  $c_l^{(n)}\in(\frac{l-1}{n!},\frac{l+1}{n!})$ such that
	\begin{align*}
		F_n&=\sum_{l=-L_n}^{L_n}\frac{E\left(\frac{l+1}{n!}\right)-E\left(\frac{l-1}{n!}\right)}{2/n!}\mathbf{1}_{s_l^{(n)}}=\sum_{l=-L_n}^{L_n} E'(c_l^{(n)})\mathbf{1}_{s_l^{(n)}}.
	\end{align*}
	Using $M_1\in\mathbb{R}$ as any bound on $E'',$ we can use the Mean Value Theorem again to establish that $|E'(t)-E'(c_l^{(n)})|\leq\frac{2M_1}{n!}$ whenever $t\in s_l^{(n)},$ thus
	\begin{align*}
		\|E'-F_n\|^2&=\|E'\mathbf{1}_{\mathbb{R}\setminus[-L_n,L_n[}\|^2+\sum_{l=-L_n}^{L_n}\|(E'-E'(c_l^{(n)}))\mathbf{1}_{s_l^{(n)}}\|^2\\
		&\leq \|E'\mathbf{1}_{\mathbb{R}\setminus[-L_n,L_n[}\|^2+4M_1^2\frac{2L_n+1}{n!^3}.
	\end{align*}
	
	Therefore, since $E'\in H$, we have $E'=\lim_{n\rightarrow\infty} F_n,$ and so $-iE'=\lim_{n\rightarrow\infty}A_n E_n.$
	
	For the last part, we calculate, for any $f_k^{(n)}\in\Omega_n,$
	\begin{align*}
		\LL  E^{(n)}, f_k^{(n)} \RR&=\frac{1}{n!\sqrt{2n!}}\sum_{l=-L_n}^{L_n}e^{-2\pi i l\frac{k}{2n!^2}}E\left(\frac{l}{n!}\right)+\frac{i}{n!^2\sqrt{2n!}}\\
		&=\frac{1}{n!\sqrt{2n!}}\left(E(0)+\frac{i}{n!}+\sum_{l=1}^{L_n}\left(e^{-2\pi i l\frac{k}{2n!^2}}E\left(\frac{l}{n!}\right)+e^{2\pi i l\frac{k}{2n!^2}}E\left(\frac{-l}{n!}\right)\right)\right)\\
		&=\frac{1}{n!\sqrt{2n!}}\left(\left(\frac{2}{\pi}\right)^{\frac{1}{4}}+\frac{i}{n!}+2\sum_{l=1}^{L_n}\cos\left(2\pi l\frac{k}{2n!^2}\right)E\left(\frac{l}{n!}\right)\right).
	\end{align*}
	
	Thus, $\operatorname{Im}(\LL E^{(n)},f_k^{(n)}\RR)=\frac{i}{n!^2\sqrt{2n!}}\neq 0,$ and so $\LL E^{(n)}, f_k^{(n)}\RR\neq 0.$
\end{proof}

Then, let $(q_j)_{j\in\mathbb{N}}$ be a count of $\mathbb{Q},$ with $q_1=0.$ Furthermore, for $r\in\mathbb{R},$ let $\mathcal{T}_r:H\rightarrow H$ be the translation by $r$, so that $(\mathcal{T}_r g)(t)=g(t-r).$ We note that $\mathcal{T}_r$ is a unitary map on $H$. Then, for $j\in\mathbb{N},$ we define $g_j=\mathcal{T}_{q_j} E$. We have that each $g_j\in H$ is smooth, with $\|g_j\|=\|E\|=1.$

Using the Fourier Transform, one could prove the following. We do find relevant to include the schema of an elementary proof, which is taken from \cite{nonez2024spectralequivalencesnonstandardsamplings}.

\begin{prop}
	The set $\{g_j\}_{j\in\mathbb{N}}$ is dense-spanning in $L_2(\mathbb{R},\mu).$
\end{prop}
\begin{proof}[Idea of Proof]
	For this part (and this part alone) we will denote a given function $f:\mathbb{R}\rightarrow\mathbb{C}$ by $f(x)$.  Let $\tilde{H}=\overline{\Span(\{g_j\}_{j\in\mathbb{N}})}\subset H.$ We also define, for $f:\mathbb{R}\rightarrow \mathbb{C},$ $P_k f$ to be the $2k$-periodic extension of $f|_{[-k,k)}.$ Since $g_j(x)=E(x)\cdot e^{2q_j x}e^{-q_j^2},$ $\tilde{H}$ must contain:
	\begin{itemize}
		\item $E(x)\cdot e^{qx}$, for any $q\in \mathbb{Q}$;
		\item $E(x)\cdot e^{tx}$, for any $t\in\mathbb{R}$, using $E(x)\cdot e^{q_nx}$ when $q_n\rightarrow t$;
		\item $E(x)\cdot e^{tx} x$, for any $t\in\mathbb{R}$, using $E(x)\cdot \frac{e^{(t+1/n)x}-e^{tx}}{1/n}$;
		\item $E(x)\cdot e^{tx} x^k$, for any $t\in\mathbb{R}$ and $k\in\mathbb{N}$, inductively using a similar sequence;
		\item  $E(x)\cdot e^{itx}$ for any $t\in\mathbb{\mathbb{R}}$, using $E(x)\cdot \left(\sum_{k=1}^n\frac{(it)^k}{k!}x^k\right);$
		\item $E(x)\cdot (P_kf)(x)$ for any $k\in\mathbb{N}$ and $f\in C_c^{\infty}(\mathbb{R})$ with $\operatorname{supp}(f)\subset(-k,k)$, using Fourier series for periodic functions $E(x)\cdot \sum_{l=-n}^n\widehat{(P_kf)}_le^{ \frac{2\pi i l}{2k}x}$;
		\item $f$ for any $f\in C_c^{\infty}(\mathbb{R})$, using $\left(\frac{2}{\pi}\right)^{\frac{1}{4}}E(x)\cdot (P_n\frac{f}{E})(x)$ for $\operatorname{supp}(f)\subset(-n,n)$;
		\item $f$ for any $f\in H$, by density of $C_c^{\infty}(\mathbb{R}).$
	\end{itemize}
	In all steps between the second and second-to-last, the Dominated Convergence Theorem is used, noting that $E\cdot f_n$ is always dominated in $L_2(\mathbb{R},\mu)$ whenever there is some real $t$ and constant $K$ for which $|f_n(x)|\leq Ke^{|tx|}$.  
\end{proof}
\begin{remark}\label{remark_dense_gaussian}
	Going through the last part of the reasoning, we also see that $\{e^{it\cdot} E\}_{t\in\mathbb{R}}$ is dense-spanning in $L_2(\mathbb{R},\mu).$ Furthermore, we can replace $E(x)$ by any function of the for $Ce^{-a x^2}$ where $a>0.$
\end{remark}
We can now define our scale. For $n\in\mathbb{N},$ let $N_n=n!,$ and for $j\leq N_n,$ let $e_j^{(n)}=RS_n^{\lfloor n!q_j\rfloor} E^{(n)},$ and $c_j^{(n)}=\frac{1}{2^j(1-2^{-N_n})\|E^{(n)}\|^2}.$ Finally, we define $(S_n)_{n\in\mathbb{N}}$ with $S_n=(H_n,A_n,\Omega_n, (e_j^{(n)},c_j^{(n)})_{j=1}^{N_n})_{n\in\mathbb{N}}.$


\begin{prop}
	We have that $(S_n)_{n\in\mathbb{N}}$ is a strong $A$-sampling-scale sequence. Furthermore, for all $j\in\mathbb{N},$  $e_j=g_j$ and $c_j=\frac{1}{2^j}$.
\end{prop}
\begin{proof}
	
	\textbf{Property (\ref{property_scale_elements}), (\ref{property_scale_infinity})  and (\ref{property_compat_H})}: It is clear that $N_n=n!\in\mathbb{N},$ and that $\lim_{n\rightarrow\infty} N_n=\infty.$ Furthermore, $c_j^{(n)}>0$  holds by inspection. Since $E^{(n)}\in H_n\setminus\{0\}$, and $RS_n$ is an isometry on $H_n$, $e_j^{(n)}\in H_n$ is well-defined and non-zero. We note $\|e_j^{(n)}\|=\|E^{(n)}\|$ whenever $j\leq N_n$.
	
	\textbf{Property (\ref{property_scale_proba})}: We can calculate $\sum_{j=1}^{N_n}c_j^{(n)}\|e_j^{(n)}\|^2=\frac{1}{1-2^{-N_n}}\sum_{j=1}^{N_n}\frac{1}{2^j}=1.$
	
	
	\textbf{Properties (\ref{property_scale_limits}), (\ref{property_compat_strong}) and (\ref{property_compat_A})}: Since $\lim_{n\rightarrow\infty}\|E^{(n)}\|=\|E\|=1,$   $\lim_{n\rightarrow\infty} c_j^{(n)}=\frac{1}{2^j}=c_j$ holds for any $j\in\mathbb{N}.$ Next, we inductively have that for any $m\in\mathbb{Z}$ and $s_l^{(n)}$ for which $|m|+|l|<n!^2,$ $$RS_n^m\mathbf{1}_{s_l^{(n)}}=\mathbf{1}_{s_{l+m}^{(n)}}=\mathcal{T}_{\frac{m}{n!}}\mathbf{1}_{s_l^{(n)}}.$$ Furthermore, we have that $\lim_{n\rightarrow\infty}n!-\frac{L_n+1}{n!}=\lim_{n\rightarrow\infty}n!(1-\frac{L_n+1}{n!^2})=\infty$. We also have that for any $q\in\mathbb{Q},$ any large enough $n\in\mathbb{N}$ leads to $\lfloor n!q \rfloor=n!q\in\mathbb{Z}.$ Thus, for any $j\in\mathbb{N},$ there exists $M\in\mathbb{N}$ such that for any $n>M,$ we have  $N_n\geq j$, $\lfloor n!q_j\rfloor=n!q_j$ and $n!|q_j|+L_n+1<n!^2.$ For such $n$, we then have
	\begin{align*}
		e_j^{(n)}&=RS_n^{n!q_j}\sum_{l=-L_n}^{L_n}E\left(\frac{l}{n!}\right)\mathbf{1}_{s_l^{(n)}}=\sum_{l=-L_n}^{L_n}E\left(\frac{l}{n!}\right)\mathcal{T}_{\frac{n!q_j}{n!}}\mathbf{1}_{s_{l}^{(n)}}=\mathcal{T}_{q_j}E^{(n)}.
	\end{align*}
	 Thus, $\|e_j^{(n)}-g_j\|=\|\mathcal{T}_{q_j}(E^{(n)}-E)\|=\|E^{(n)}-E\|$ for any such large $n$, which implies $g_j=\lim_{n\rightarrow\infty} e_j^{(n)}=e_j$  by Proposition \ref{proposition_gaussian_finite}.
	 
	 We also have that since $LS_n$ and $RS_n$ are inverses of each other, both commute with $A_n$. Thus, we have that $A_ne_j^{(n)}=RS_n^{\lfloor n!q_j\rfloor}A_nE^{(n)}.$ Since $A_nE^{(n)}\in\Span(\{\mathbf{1}_{s_l^{(n)}}\}_{l=-L_n-1}^{L_n+1}),$ we also have that $RS_n^{\lfloor n!q_j\rfloor}A_nE^{(n)}=\mathcal{T}_{q_j}A_nE^{(n)}$ for any $n>M.$ Therefore, $A_n e_j^{(n)}$ converges to $-i\mathcal{T}_{q_j}E'$ by Proposition \ref{proposition_gaussian_finite}. 
	 
	 We finally note that since $A_nE^{(n)}$ is a convergent thus bounded sequence, and since $\|A_ne_j^{(n)}\|=\|A_nE^{(n)}\|$ holds whenever $j\leq N_n$, and since $\sum_{j=1}^{N_n}c_j^{(n)}=\frac{1}{\|E^{(n)}\|^2}$ is also bounded, we have that $\sum_{j=1}^{N_n}c_j^{(n)}\|A_ne_j^{(n)}\|^2$ is a bounded sequence.
	
	\textbf{Property (\ref{property_scale_dense})}: Density of $\{e_j\}_{j\in\mathbb{N}}$ is given by the previous proposition.
	
	\textbf{Property (\ref{property_scale_bias})}: We have $\sum_{j\in\mathbb{N}}c_j\|e_j\|^2=\sum_{j\in\mathbb{N}}\frac{1}{2^j}=1.$
	
	
	\textbf{Property (\ref{property_compat_Omega}}): For $f\in\Omega_n,$ $\LL e_1^{(n)}, f\RR=\LL E^{(n)},f\RR\neq 0$ by Proposition \ref{proposition_gaussian_finite}.
\end{proof}

\subsection{Resulting measures}

We now study the resulting measure space induced by this $A$-sampling-scale sequence.
\begin{definition}
	Let $\varphi:\mathbb{R}\rightarrow\mathbb{R}_{>0}$ be defined by $\varphi(\omega)=\left(\frac{\pi}{2}\right)^{\frac{1}{2}}e^{-\frac{\pi^2\omega^2}{2}},$ and let $\mu'$ be the probability measure on $\Borel(\mathbb{R})$ such that $d\mu'=\varphi d\mu.$ Furthermore, let $X:\mathbb{R}\rightarrow Q$ such that $$X(\omega)=\left(\frac{1}{\sqrt{2^{j}}}e^{-\pi i q_j \omega}\right)_{j\in\mathbb{N}}.$$
\end{definition}
\begin{theorem}\label{theorem_differential_fourier_transform}
	The sequence $(S_n)_{n\in\mathbb{N}}$ is convergent, with the resulting $\HMU$ being given by the pushforward measure $\HMU=X_{\#}(\mu').$ Furthermore, if $U:H\rightarrow L_2(\mathbb{R},\mu')$ is given by $U(g)=\HU(g)\circ X,$ then $U$ is a unitary map and for any $\psi\in L_2(\mathbb{R},\mu)\cap L_1(\mathbb{R},\mu)$, $$(U(\psi))(\omega)=\frac{1}{\sqrt{2\varphi(\omega)}}\int_{\mathbb{R}}e^{-\pi i \omega t} \psi(t)d\mu(t)$$ for almost all $\omega\in \mathbb{R}.$ Finally, if $m:\mathbb{R}\rightarrow\mathbb{R}$ is given by $m=\Hm\circ X,$ then $m(\omega)=\pi \omega$ almost everywhere in $\mathbb{R}$, and for any $z\in \Dom(A)=C_c^{\infty}(\mathbb{R})$, $$(U(Az))=m\cdot U(z).$$
\end{theorem}

Thus, the resulting  isometry is, up to normalization, the Fourier Transform when it is well defined. In fact, one could extend the Fourier Transform $\mathcal{F}:H\rightarrow H$ with $$(\mathcal{F}(\psi))(\omega)=\sqrt{2\varphi(2\omega)}\cdot(U(\psi))(2\omega).$$
Plancherel's Theorem follows directly from $\int_{\mathbb{R}}|U(g)|^2d\mu'=\int_{\mathbb{R}}|g|^2d\mu$, while the differentiation formula follows directly from the previous theorem, at least on $C_c^{\infty}(\mathbb{R})$ and the domain of its closure given by the Sobolev space $\Dom(\bar{A})=H^1(\mathbb{R})$.

As with Theorem \ref{theorem_shift_fourier_series}, we will only outline the important steps of the proof, as we think the full calculations are technical and needless. The reader may complete the details at their convenience. Just like with Theorem \ref{theorem_shift_fourier_series}, a complete proof may be found in  Appendix \ref{appendix_differential}. However, that proof uses nonstandard techniques to link with \cite{nonez2024spectralequivalencesnonstandardsamplings}.

\begin{proof}[Key Steps for the Proof of Theorem \ref{theorem_differential_fourier_transform}]
	In order, we establish the following:
	\begin{itemize}
		\item Using direct calculations, we have, for any $f_k^{(n)}\in\Omega_n,$  that $\mu_n(f_k^{(n)})=\frac{|\LL E^{(n)}, f_k^{(n)}\RR|^2}{\|E^{(n)}\|^2}.$ Furthermore, for any $j\leq N_n$, we have that $$(U_n(e_j^{(n)}))(f_k^{(n)})=e^{-\pi i \frac{\lfloor n!q_j\rfloor}{n!}\frac{k}{n!}}\frac{\LL E^{(n)}, f_k^{(n)}\RR}{|\LL E^{(n)}, f_k^{(n)}\RR|}\|E^{(n)}\|$$ and that $\pi_j (X_n(f_k^{(n)}))=\frac{1}{\sqrt{1-2^{-N_n}}}\frac{1}{\sqrt{2}^j}e^{-\pi i \frac{\lfloor n!q_j\rfloor}{n!}\frac{k}{n!}}\frac{\LL E^{(n)}, f_k^{(n)}\RR}{|\LL E^{(n)}, f_k^{(n)}\RR|}.$
		\item For $n\in\mathbb{N}$, we define $\psi_n:\mathbb{R}\rightarrow\mathbb{K}$ with $$\psi_n(\omega)=\begin{cases}
			\sqrt{n!}\LL E^{(n)},f_{k_n(\omega)}^{(n)}\RR \quad \text{if $\;\omega\in[-n!,n![$,}\\
			0 \quad \text{ otherwise, }
		\end{cases}$$ where $k_n(\omega)$ is the integer in $[-n!^2,n!^2[$ for which $\omega\in[\frac{k_n(\omega)}{n!},\frac{k_n(\omega)+1}{n!}[.$ Using Dominated Convergence Theorem, we have that for $\omega\in\mathbb{R},$ the limit $$\lim_{n\rightarrow\infty}\psi_n(\omega)=\frac{1}{\sqrt{2}}\int_{\mathbb{R}}e^{-\pi i t\omega}E(t)d\mu(t)=\left(\frac{\pi}{2}\right)^{\frac{1}{4}}e^{-\frac{\pi^2\omega^2}{4}}=\sqrt{\varphi(\omega)}$$ holds. 
		\item Since that last quantity is a positive real number, we can deduce that for $\omega\in\mathbb{R}$, $\lim_{n\rightarrow\infty} X_n(f_{k_n(\omega)}^{(n)})=\left(\frac{1}{\sqrt{2}^j}e^{-\pi i q_j \omega}\right)_{j\in\mathbb{N}}=X(\omega).$ 
		\item We can also establish that $(\psi_n)_{n\in\mathbb{N}}$ is uniformly bounded.
		\item We define, for $r>0,$ $V_r^{(n)}=\{f_k^{(n)}\in \Omega_n\;|\; -n!r\leq k <n!r\}.$  Then, we can directly establish the following property, using Dominated Convergence Theorem.
		\item \textbf{Property 1}. Suppose that $(g_n:\Omega_n\rightarrow\mathbb{C})_{n\in\mathbb{N}}$ and $(G_n:\mathbb{R}\rightarrow\mathbb{C})_{n\in\mathbb{N}}$ are function sequences such that for any $r>0,$ $\int_{V_r^{(n)}}g_nd\mu_n=\int_{[-r,r[} G_nd\mu.$ Suppose further that $(G_n)_{n\in\mathbb{N}}$ is uniformly bounded on any compact subset $K\subset\mathbb{R}$, and pointwise convergent to Lebesgue-integrable $G:\mathbb{R}\rightarrow\mathbb{C}.$ Then, for any $r>0$, $\lim_{n\rightarrow\infty}\int_{V_r^{(n)}}g_nd\mu_n=\int_{[-r,r]} Gd\mu.$
		\item Through straightforward calculations, we have that Property 1 applies to the case $g_n=1$, $G_n=\frac{1}{\|E^{(n)}\|^2}|\psi_n|^2,$ with $G=\varphi,$ and thus for any $r>0$ $\lim_{n\rightarrow\infty}\mu_n(V_r^{(n)})=\int_{[-r,r[}\varphi d\mu.$ Since $\lim_{r\rightarrow\infty}\int_{[-r,r[}\varphi d\mu=1,$ since all $\mu_n$ are probability measures and since $V_{n!}^{(n)}=\Omega_n$ holds for any $n$, we have $\mu_n(\Omega_n\setminus V_r^{(n)})\rightarrow 0$ as $r\rightarrow\infty$ uniformly, in the sense that for any $\epsilon>0,$ there exists $R>0$ such that for any $r>R$ and $n\in\mathbb{N},$  $\mu_n(\Omega_n\setminus V_r^{(n)})<\epsilon.$ We can thus show, through Cauchy-Schwarz inequality, the following property.
		\item \textbf{Property 2}. If $(g_n)_{n\in\mathbb{N}},$ $(G_n)_{n\in\mathbb{N}})$ and $G$ are as in Property 1, and if furthermore $\int_{\Omega_n}|g_n|^2d\mu_n$ is uniformly bounded, then $$\lim_{n\rightarrow\infty} \int_{\Omega_n} g_nd\mu_n=\int_{\mathbb{R}} Gd\mu.$$ 
		\item We then establish the convergence of the $A$-sampling-scale sequence. Given $h\in C(Q),$ Property 2 applies to $g_n=h\circ X_n$, $G_n$ given by $G_n(\omega)=\frac{1}{\|E^{(n)}\|^2}h(X_n((f_{k_n(\omega)}^{(n)}))|\psi_n(\omega)|^2$ and $G=(h\circ X)\cdot \varphi.$ Therefore, we have $$\lim_{n\rightarrow\infty}\int_{\Omega_n}(h\circ X_n)d\mu=\int_{\mathbb{R}} (h\circ X)\cdot\varphi d\mu=\int_{\mathbb{R}} (h\circ X)d\mu'.$$ Since  $h\in C(Q)$ is arbitrary, this establishes the convergence of $(S_n)_{n\in\mathbb{N}},$ with $\HMU=X_{\#}(\mu').$ 
		\item We have that for any $j$, $V_j(X(\omega))=e^{-\pi i q_j \omega}$. We can show that $\{V_j\}_{j\in\mathbb{N}}$ is dense-spanning in $L_2(\mathbb{R},\mu')$ by using Remark \ref{remark_dense_gaussian} to show that $\{V_j\cdot\varphi\}$ is dense-spanning in $L_2(\mathbb{R},\mu.)$ This establishes that $U$ is unitary. This also shows that elements of the form $h\circ X$ with $h\in C(Q)$ are dense in $L_2(\mathbb{R},\mu').$
		\item To establish that $$(U(\psi))(\omega)=\frac{1}{\sqrt{2\varphi(\omega)}}\int_{\mathbb{R}}e^{-\pi i \omega t} \psi(t)d\mu(t)$$ for almost all $\omega\in \mathbb{R}$ given $\phi\in L_2(\mathbb{R},\mu)\cap L_1(\mathbb{R},\mu)$, we consider the sequence given by $\phi_n=\Proj_{H_n}\phi.$ One one hand, by Proposition \ref{proposition_converging_isometry}, we know that for any $h\in C(Q),$ $$\lim_{n\rightarrow\infty}\int_{\Omega_n} U_n(\phi_n)\cdot (h\circ X_n)d\mu_n=\int_{Q}\HU(\phi)\cdot hd\HMU=\int_{\mathbb{R}}U(\phi)\cdot (h\circ X)\cdot\varphi d\mu.$$ On the other hand, noting $\int_{\Omega_n}|U_n(\phi_n)|^2d\mu_n=\|\phi_n\|^2\leq \|\phi\|^2,$ we have that Property 2 applies to $g_n=U_n(\phi_n)\cdot(h\circ X_n)$ and $G_n$ being given by $G_n(\omega)=\frac{1}{\|E^{(n)}\|}\sqrt{n!}\LL \phi, f_{k_n(\omega)}^{(n)}\RR\cdot h(X_n(f_{k_n(\omega)}^{(n)})\cdot|\psi_n(\omega)|.$ Stemming from the Dominated Convergence theorem and the integrability of $\phi$ on $\mathbb{R}$, the pointwise limit $G$ of $G_n$ is given by $$G(\omega)=\frac{1}{\sqrt{2}}F_{\phi}(\omega)\cdot h(X(\omega))\cdot\sqrt{\varphi(\omega)},$$ where $F_{\phi}(\omega)=\int_{\mathbb{R}}e^{-\pi i \omega t}\phi(t)d\mu(t).$ We thus get $$\lim_{n\rightarrow\infty}\int_{\Omega_n} U_n(\phi_n)\cdot (h\circ X_n)d\mu_n=\int_{\mathbb{R}}Gd\mu,$$ and so $$\int_{\mathbb{R}}U(\phi)\cdot (h\circ X)\cdot\varphi d\mu=\int_{\mathbb{R}}\frac{1}{\sqrt{2}}F_{\phi}\cdot (h\circ X)\cdot \sqrt{\varphi}d\mu.$$ By considering complex measure $\nu$ with $d\nu=(U(\phi)\varphi-\frac{1}{\sqrt{2}}F_{\phi}\sqrt{\phi})d\mu,$ the last equation becomes $\int_Qh dX_{\#}(\nu)=0$ for any $h\in C(Q).$ By uniqueness given by the Riesz-Markov Theorem, we have $X_{\#}(\nu)=0,$ and since $X$ is injective, we have $\nu=0.$ Thus, for almost all $\omega\in\mathbb{R},$ we have $$(U(\phi))(\omega)=\frac{1}{\sqrt{2\phi}}\int_{\mathbb{R}}e^{-\pi i \omega t}\phi(t)d\mu(t).$$
		\item That $U(Az)=m\cdot U(z)$ directly follows from $\HU(Az)=\Hm\cdot\HU(z)$ for any $z\in\Dom(A).$ The last step is thus to show that $m(\omega)=\pi\omega$ for almost all $\omega\in\mathbb{R}.$ Thanks to Remark \ref{remark_convergent_eigenvalue}, we know that for any $h\in C(Q),$ $$\lim_{n\rightarrow\infty} \int_{\Omega_n}\lambda_n\cdot (h\circ X_n)d\mu_n=\int_Q \Hm\cdot hd\HMU=\int_{\mathbb{R}}m\cdot (h\circ X) d\mu'.$$ Noting that $\int_{\Omega_n}|\lambda_n|^2d\mu_n=\frac{\|A_n E^{(n)}\|^2}{\|E^{(n)}\|^2}$, which is bounded as a convergent sequence, we have that Property to applies to $g_n=\lambda_n\cdot (h\circ X_n),$ $G_n$ being given by $G_n(\omega)=\frac{1}{\|E^{(n)}\|^2}n! \sin(\pi\frac{k_n(\omega)}{n!^2})\cdot h(X_n(f_{k_n(\omega)}))\cdot|\psi_n|^2$ and $G$ given by $G(\omega)=\pi\omega \cdot h(X(\omega))\cdot\varphi.$ Thus, we have $\lim_{n\rightarrow\infty} \int_{\Omega_n}\lambda_n\cdot (h\circ X_n)d\mu_n=\int_{\mathbb{R}}\pi\omega\cdot h(X(\omega))\cdot\varphi d\mu$, and so $$\LL m, h\circ X\RR_{\mu'}=\int_{\mathbb{R}}m\cdot (h\circ X) d\mu'=\int_{\mathbb{R}}\pi\Id_{\mathbb{R}}\cdot (h\circ X)d\mu'=\LL \pi\Id_{\mathbb{R}}, h\circ X \RR_{\mu'}. $$ By density of $h\circ X$, we conclude that $m=\pi\Id_{\mathbb{R}}$ in $L_2(\mathbb{R},\mu').$
	\end{itemize}
\end{proof}
\appendix
\section{Nonstandard characterization}\label{appendix_nonstandad}

In this section, we connect the work done in \cite{nonez2024spectralequivalencesnonstandardsamplings} with the work done here, in order to simplify some of the technicalities with the examples and avoid redundancy. In this section, we will use nonstandard techniques and terminology. The uninitiated reader may familiarize themselves using \cite{nsarecent} or \cite{arkeryd2012nonstandard}. We further encourage familiarity with the techniques used in \cite{nonez2024spectralequivalencesnonstandardsamplings} or \cite{goldbring2025nonstandardapproachdirectintegral}. 

We now assume that ${}^*$ is an countably saturated extension of some superstructure containing at least $H$. When $o$ is an standard object or function, we may note the extension ${}^*o$ by $o$  to lighten the notation. However, for standard sets, we will not use the same shortcut to avoid any ambiguities (for example, in "for any $\epsilon\in\mathbb{R}_{>0}$" we mean for any standard such $\epsilon>0$).

We first begin by restating the definition of the \textbf{nonstandard} sampling and compatible standard-biased scale found as in \cite{nonez2024spectralequivalencesnonstandardsamplings}.

\begin{defn}\label{definition_nonstandard_sampling}
		The internal triplet $(\tilde{H},\tilde{A},\tilde{\Omega})$ is a sampling for $A$ if:
	
	\begin{enumerate}
		\item \label{property_nonstandard_dimension} $\tilde{H}<{}^*H$, and ${}^*\operatorname{dim}(\tilde{H})\in{}^*\mathbb{N}$. 
		\item \label{property_nonstandard_symmetric} $\tilde{A}:\tilde{H}\rightarrow\tilde{H}$ is a ${}^*$linear symmetric operator on $\tilde{H}$.
		\item \label{property_nonstandard_eigenbasis} $\tilde{\Omega}\subset\tilde{H}$ is an orthonormal ${}^*$eigenbasis of $\tilde{A}$.
		
		\item \label{property_nonstandard_approximation} $G(A)\subset$ st$(G(\tilde{A}))$. In other words, for any $x\in \operatorname{dom}(A)$, there exists $\tilde{x}\in \tilde{H}$ such that $x=$st$(\tilde{x})$ and $Ax=$st$(\tilde{A}\tilde{x})$.
		\setcounter{pointnumber}{\value{enumi}}
	\end{enumerate}
	
	Furthermore, given infinite $\tilde{N}\in{}^*\mathbb{N},$ the hyperfinite sequence in ${}^*H\times {}^*\mathbb{R}_{\geq0}$ given by $(\tilde{e}_j,\tilde{c}_j)_{j=1}^{\tilde{N}}$ is a standard biased scale if:
	
	\begin{enumerate}
		\setcounter{enumi}{\value{pointnumber}}
		\item\label{property_nonstandard_scale_NS} for each $j\in\mathbb{N}$, both $\tilde{e}_j$ and $\tilde{c}_j$ are nearstandard with $\St(\tilde{e}_j)\in H\setminus\{0\}$ and $\St(\tilde{c}_j)\in\mathbb{R}_{>0}.$
		\item\label{property_nonstandard_scale_dense} $(\St(\tilde{e}_j))_{j\in\mathbb{N}}$ spans a dense subset in $H$.
		\item \label{property_nonstandard_scale_proba} $\sum_{j=1}^{\tilde{N}}\tilde{c}_j\|\tilde{e}_j\|^2=1$.
		\item \label{property_nonstandard_scale_bias} $\sum_{j\in\mathbb{N}}\St(\tilde{c}_j)\|\St(\tilde{e}_j)\|^2=1$.
		
		\setcounter{pointnumber}{\value{enumi}}
	\end{enumerate}
	
	Finally, we say that this scale is compatible with sampling $(\tilde{H},\tilde{A},\tilde{\Omega})$ if :
	\begin{enumerate}
		\setcounter{enumi}{\value{pointnumber}}
		\item \label{property_nonstandard_compat_H} For each $j\leq\tilde{N}$, $\tilde{e}_j\in\tilde{H}$.
		\item \label{property_nonstandard_compat_A} For each $j\in \mathbb{N}$, $\tilde{A}\tilde{e}_j$ is nearstandard.
		\item \label{property_nonstandard_compat_Omega} For each $f\in\tilde{\Omega}$, there exists $j\leq\tilde{N}$ such that $\tilde{c}_j>0$ and $(\tilde{e}_j,f)\neq0$.
	\end{enumerate}
\end{defn}

We can now show the first connection between sampling-scale-sequences and nonstandard samplings. It is an intuitive one, as it was the latter that inspired the former.

\begin{prop}
	Let $(S_n)_{n\in\mathbb{N}}=(H_n,A_n,\Omega_n, (e_j^{(n)},c_j^{(n)})_{j=1}^{N_n})_{n\in\mathbb{N}}$ be a \textbf{strong} A-sampling-scale sequence. Then, for any infinite $K\in{}^*\mathbb{N},$ we have that $(H_K, A_K, \Omega_K)$ is a sampling for $A$ for which $(e_j^{(K)},c_j^{(K)})_{j=1}^{N_K}$ is a compatible standard-biased scale, with $e_j=\St(e_j^{(K)})$ and $c_j=\St(c_j^{(K)}).$
\end{prop}
\begin{proof}
	We first show that $(H_K, A_K, \Omega_K)$ is a sampling for $A$.
	
	\textbf{Properties (\ref{property_nonstandard_dimension}), (\ref{property_nonstandard_symmetric}) and (\ref{property_nonstandard_eigenbasis})}: Direct application of the Transfer Principle to Properties (\ref{property_sampling_finite_dim}),(\ref{property_sampling_symmetric_op}) and (\ref{property_sampling_eigenbasis})  of Definition \ref{definition_sampling} respectively.
	
	\textbf{Property (\ref{property_nonstandard_approximation})}: For $x\in\Dom(A)$, let $(x_n)_{n\in\mathbb{N}}$ be a sequence given by Property (\ref{property_sampling_approx}) of Definition \ref{definition_sampling}, so that each $x_n\in H_n$, $x=\lim_{n\rightarrow\infty}x_n=x$ and $Ax=\lim_{n\rightarrow\infty}A_nx_n.$ Let $\tilde{x}={}^*x_K$. Then, by Transfer Principle, we have $\tilde{x}\in H_K$. Since $K$ is infinite we obtain $\tilde{x}\simeq x$ and $A_K \tilde{x}\simeq Ax$ by nonstandard characterization of limits.
	
	Then, we show that $(e_j^{(K)},c_j^{(K)})_{j=1}^{N_K}$ is a standard-biased scale. We first note that for all $n\in\mathbb{N}$ and $j\leq N_n$, $(e_j^{(n)},c_j^{(n)})\in H\times \mathbb{R}_{>0}$ by Property (\ref{property_scale_elements}) of Definition \ref{definition_scale}. Thus, by Transfer Principle, for any $j\leq N_K,$ $(e_j^{(K)},c_j^{(K)})\in {}^*H\times {}^*\mathbb{R}_{>0}$. We also note that by Property (\ref{property_scale_infinity}) of the same definition, $N_K$ is infinite.
	
	\textbf{Property (\ref{property_nonstandard_scale_NS})}: By Property (\ref{property_scale_limits}) of Definition \ref{definition_scale}, we have $e_j=\lim_{n\rightarrow\infty}e_j^{(n)}$ and $c_j=\lim_{n\rightarrow\infty} c_j^{(n)}.$ By nonstandard characterization of limits, $e_j^{(K)}\simeq e_j\in H\setminus\{0\}$ and $c_j^{(K)}\simeq c_j\in\mathbb{R}_{>0}.$
	
	\textbf{Properties (\ref{property_nonstandard_scale_dense}) and (\ref{property_nonstandard_scale_bias})} are thus the respective statements of Properties (\ref{property_scale_dense}) and (\ref{property_scale_bias}).
	
	\textbf{Property (\ref{property_nonstandard_scale_proba})}: Direct application of the Transfer Principle to Property (\ref{property_scale_proba}) of Definition \ref{definition_scale}.
	
	Now, we conclude the proof by showing compatibility.
	
	\textbf{Properties (\ref{property_nonstandard_compat_H}) and (\ref{property_nonstandard_compat_Omega})}: Direct application of the Transfer Principle to Properties (\ref{property_compat_H}) and (\ref{property_compat_Omega}) of Definition \ref{definition_sampling_scale_sequence}. We note that $c_j^{(K)}>0$ for all $j\leq N_K$.
	
	\textbf{Property (\ref{property_nonstandard_compat_A})}: Since the $A$-sampling-scale sequence is strong, we have Property (\ref{property_compat_strong}) of Definition \ref{definition_sampling_scale_sequence}, and so for any $j\in\mathbb{N},$ $\lim_{n\rightarrow\infty}A_n e_j^{(n)}\in H$ exists. By nonstandard characterization of limits, we conclude $A_Ke_j^{(K)}$ is nearstandard for any such $j\in\mathbb{N}$, with $\St\left(A_K e_j^{(K)}\right)=\lim_{n\rightarrow\infty}A_n e_j^{(n)} $.
\end{proof}

We now fix $(S_n)_{n\in\mathbb{N}}=(H_n,A_n,\Omega_n, (e_j^{(n)},c_j^{(n)})_{j=1}^{N_n})_{n\in\mathbb{N}}$ as any strong $A$-sampling-scale sequence. We do not assume its convergence. 

Then, given infinite $K\in{}^*\mathbb{N}$, we define the main induced objects from the nonstandard sampling $(H_K,A_K,\Omega_K)$ and the compatible scale $(e_j^{(K)}, c_j^{(K)})_{j=1}^{N_K}$, as constructed in \cite{nonez2024spectralequivalencesnonstandardsamplings}.

\begin{defn}\label{definition_nonstandard_internal_measure}
	Let $\tilde{\mu}^{(K)}:{}^*\mathcal{P}(\Omega_K)\rightarrow{}^*\mathbb{R}_{\geq0}$ be the internal probability measure given by $\tilde{\mu}^{(K)}(V)=\sum_{j=1}^{N_K}c_j^{(K)}\|\Proj_{\Span(V)}e_j^{(K)}\|^2$, and let $(\Omega_K,\mathcal{A}_L^{(K)},\mu_L^{(K)})$ be the induced Loeb probability space. Furthermore, let $\tilde{U}^{(K)}:{}^*H\rightarrow {}^*L_2(\Omega_K,\tilde{\mu}^{(K)})$ be given by $\left(\tilde{U}^{(K)}(x)\right)(f)=\frac{\LL x, f\RR}{\sqrt{\tilde{\mu}^{(K)}(f)}},$ where $\tilde{\mu}^{(K)}(f)=\tilde{\mu}^{(K)}(\{f\})>0.$ Then, let $\tilde{d}^{(K)}$ be the pseudometric on $\Omega_K$ given by $$\tilde{d}^{(K)}(f_1, f_2)=\sum_{j=1}^{N_K}\left(c_j^{(K)}\right)^{\frac{3}{2}}\|e_j^{(K)}\|^2\left|\left(\tilde{U}^{(K)}(e_j^{(K)})\right)(f_1)-\left(\tilde{U}^{(K)}(e_j^{(K)})\right)(f_2)\right|,$$ and let $\sim$ be the equivalence relation on $\Omega_K$ be given by $f_1\sim f_2$ when $\tilde{d}^{(K)}(f_1,f_2)\simeq 0.$ Then, let $\HOM^{(K)}=\Omega_K/\sim$, inducing the natural map $\HNU^{(K)}:\Omega_K\rightarrow\HOM^{(K)}.$
\end{defn}

We now summarize the main results of \cite{nonez2024spectralequivalencesnonstandardsamplings} applied here in the following theorems. In all of them, $K$ is any arbitrary infinite hypernatural. 

\begin{prop}\label{proposition_nonstandard_limited_isometry}
	For any finite $x\in {}^*H$, $(\tilde{U}^{(K)}(x))(f)$ is finite for almost all $f\in\Omega_K$ with respect to $\mu_L^{(K)},$ with specific uniform bound $|\tilde{U}(e_j^{(K)})|^2\leq\frac{1}{c_j^{(K)}}$ for any standard $j\in\mathbb{N}$. Furthermore, $\lambda_K:\Omega_K\rightarrow{}^*\mathbb{R},$ the eigenvalue function of $A_K$, is also $\mu_L^{(K)}$-almost-everywhere finite.
\end{prop}
\begin{thm}\label{theorem_nonstandard_spectral_Loeb}	
	The operator $U_L^{(K)}:H\rightarrow L_2(\Omega_K,\mu_L^{(K)})$ defined by $U_L^{(K)}(x)=\St\circ \left(\tilde{U}^{(K)}(x)\right)$  is a linear isometry. Furthermore, for any $x\in\Dom(A),$ we have $U_L^{(K)}(Ax)=m_L^{(K)}\cdot U_L^{(K)}(x),$ where $m_L^{(K)}=\St\circ\lambda_K.$
\end{thm}
\begin{prop}\label{proposition_nonstandard_hull_compact}	
	We have that $d^{(K)}(\HNU^{(K)}(f_1),\HNU^{(K)}(f_2))=\St(\tilde{d}^{(K)}(f_1,f_2))$ induces a well-defined metric on $\HOM^{(K)}.$ Under that metric, $\HOM^{(K)}$ a compact metric space for which the natural map $\HNU^{(K)}:(\Omega_K,\mathcal{A}_L^{(K)})\rightarrow(\HOM^{(K)},\Borel(\HOM^{(K)}))$ is a measurable function. Furthermore, given any $f_1,f_2\in\Omega_K,$ we have that $\HNU^{(K)}(f_1)=\HNU^{(K)}(f_2)$ if and only if $(\tilde{U}^{(K)}(e_j^{(K)}))(f_1)\simeq(\tilde{U}^{(K)}(e_j^{(K)}))(f_2)$ for all $j\in\mathbb{N}.$
\end{prop}
\begin{thm}\label{theorem_nonstandard_spectral_hull}
	Given the pushforward measure $\HMU^{(K)}=\HNU^{(K)}_{\#}(\mu_L^{(K)})$ on $\Borel(\HOM^{(K)}),$ the space $\hat{H}^{(K)}=L_2(\HOM^{(K)}, \HMU^{(K)})$ and the isometric map $\hat{\mathcal{I}}^{(K)}:\rightarrow \hat{H}^{(K)}$ given by $\hat{\mathcal{I}}^{(K)}(g)=g\circ\HNU^{(K)},$ we have that $U_L^{(K)}(H)\subset\hat{\mathcal{I}}^{(K)}(\hat{H}^{(K)})$, inducing the isometry $$\HU^{(K)}=(\hat{\mathcal{I}}^{(K)})^{-1}\circ U_L^{(K)} :H\rightarrow  \hat{H}^{(K)}.$$ 
	
	Furthermore, there exists a Borel measurable function $\Hm^{(K)}:\HOM^{(K)}\rightarrow \mathbb{R}$ such that $\Hm^{(K)}\circ\HNU^{(K)}\equiv m_L^{(K)}$ almost everywhere on $\Omega_K$ with respect to $\mu_L^{(K)},$ and so for any $x\in\Dom(A),$ we have that $\HU^{(K)}(Ax)=\Hm^{(K)}\cdot \HU^{(K)}(x).$
\end{thm}

We now want to properly characterize the convergence of $(S_n)_{n\in\mathbb{N}}$, as well as the resulting objects $\HMU$, $\HU$ and $\Hm$, using respectively $\HMU^{(K)}$, $\HU^{(K)}$ and $\Hm^{(K)}.$ First, from the Transfer Principle applied to both $\mu_n$ and $U_n$, and comparison of definitions, we can deduce the following.

\begin{prop}
	We have that for any infinite $K\in{}^*\mathbb{N}$, ${}^*\mu_K=\tilde{\mu}^{(K)}$ and ${}^*U_K=\tilde{U}^{(K)}.$
\end{prop}

We remind the function $X_n:\Omega_n\rightarrow Q$ for any $n\in\mathbb{N}$ as per Definition \ref{definition_finite_embedding_cube}. Furthermore, we note that since $Q$ is compact and Hausdorff, the standard part function $\St_Q:{}^*Q\rightarrow Q$ is well defined on ${}^*Q$.

\begin{definition}
	Given an infinite $K\in{}^*\mathbb{N},$ let $\Psi^{(K)}:\HOM^{(K)}\rightarrow Q$ given by $$\Psi^{(K)}(\HNU^{(K)}(f))=\St_Q((X_K(f))$$ for $f\in\Omega_K.$
\end{definition}
\begin{prop}
	The map $\Psi^{(K)}$ is a well-defined function. Furthermore, $\Psi^{(K)}$ is an homeomorphism between $\HOM^{(K)}$ and $\St_Q\left(X_K(\Omega_K)\right).$
\end{prop}
\begin{proof}
	We first prove that for any $f_1,f_2\in\Omega_K$, $\HNU(f_1)=\HNU(f_2)$ if and only if $\St_Q(X_K(f_1))=\St_Q(X_K(f_2)),$ which shows that $\Psi^{(K)}$ is well-defined and injective. 
	
	Let $j\in\mathbb{N}$. We know that by definition of $Q$, $\pi_j$ is continuous on $Q$. Thus, $\pi_j\circ \St_Q=\St\circ{}^*\pi_j.$ Furthermore, by definition of $X_K,$ and since $N_K$ is infinite and so $j\leq N_K$, we have ${}^*\pi_j\circ X_K=\sqrt{c_j^{(K)}}{}^*U_K(e_j^{(K)})=\sqrt{c_j^{(K)}}\tilde{U}^{(K)}(e_j^{(K)}).$ Therefore, we have that $\pi_j\circ \St_Q\circ X_K=\sqrt{c_j}\cdot\St\circ \tilde{U}^{(K)}(e_j^{(K)}).$ 
	
	Since each $c_j>0$, and since $j\in\mathbb{N}$ is arbitrary, we have that for any $f_1 f_2\in\Omega_K,$ the equivalence chain
	\begin{align*}
		&\phantom{{}\iff{}} \HNU^{(K)}(f_1)=\HNU^{(K)}(f_2)\\
		 &\iff \St((\tilde{U}^{(K)}(e_j^{(K)}))(f_1))=\St((\tilde{U}^{(K)}(e_j^{(K)}))(f_2)) \text{ for all $j\in\mathbb{N}$ }\\
		 &\iff \pi_j(\St_Q(X_K(f_1)))=\pi_j(\St_Q(X_K(f_2)))\text{ for all $j\in\mathbb{N}$ }\\
		 &\iff \St_Q(X_K(f_1))=\St_Q(X_K(f_2))
	\end{align*}
	holds, with the first equivalence being due to Proposition  \ref{proposition_nonstandard_hull_compact}. Thus, $\Psi^{(K)}$ is well-defined and injective. Furthermore, it is clear from the definition that $\Psi^{(K)}(\HOM^{(K)})=\St_Q(X_K(\Omega_K)).$ Therefore, since $\HOM^{(K)}$ is compact, proving that $\Psi^{(K)}$ is continuous is sufficient to establish that it is an homeomorphism between the two sets.
	
	From the fact that ${}^*\pi_j\circ X_K=\sqrt{c_j^{(K)}}\tilde{U}^{(K)}(e_j^{(K)})$ for any $j\in\mathbb{N}$ and from the definition of $\tilde{d}^{(K)}$ we have that for any $f_1,f_2\in\Omega_K,$ $$\left|({}^*\pi_j\circ X_K)(f_1)-({}^*\pi_j\circ X_K)(f_2) \right|\leq\frac{\tilde{d}^{(K)}(f_1,f_2)}{c_j^{(K)}\|e_j^{(K)}\|^2}.$$ Therefore, since $\St(c_j^{(K)}\|e_j^{(K)}\|^2)=c_j\|e_j\|^2>0,$ we have $$ \St\left(\frac{\tilde{d}^{(K)}(f_1,f_2)}{c_j^{(K)}\|e_j^{(K)}\|^2}\right)=\frac{d^{(K)}(\HNU^{(K)}(f_1),\HNU^{(K)}(f_2))}{c_j\|e_j\|^2},$$ from which we conclude, using $\pi_j\circ\Psi^{(K)}\circ\HNU^{(K)}=\St\circ{}^*\pi_j\circ X_K$, that $$\left|\pi_j(\Psi^{(K)}(\HNU^{(K)}(f_1)))-\pi_j(\Psi^{(K)}(\HNU^{(K)}(f_2)))\right|\leq\frac{d^{(K)}(\HNU^{(K)}(f_1),\HNU^{(K)}(f_2))}{c_j\|e_j\|^2}.$$ 
	
	This implies that $\pi_j\circ \Psi^{(K)}$ is Lipschitz continuous for each $j\in\mathbb{N}$. By nature of the product topology, that is sufficient to establish that $\Psi^{(K)}$ is continuous, ending the proof.
\end{proof}

\begin{definition}
	For an infinite $K\in{}^*\mathbb{N},$ let $\mu_Q^{(K)}$ be the pushforward probability measure $\mu_Q^{(K)}=\Psi^{(K)}_{\#}(\HMU^{(K)})$ on $\Borel(Q).$ Noting that $\mu_Q^{(K)}(Q\setminus\Psi^{(K)}(\HOM^{(K)}))=0,$ let $U_Q^{(K)}:H\rightarrow L_2(Q,\mu_Q^{(K)})$ with $(U_Q^{(K)}(x))|_{\Psi^{(K)}(\HOM^{(K)})}=(\HU^{(K)}(x))\circ (\Psi^{(K)})^{-1}.$ Finally, let $m_Q^{(K)}=Q\rightarrow\mathbb{R}$ measurable with $m_Q^{(K)}|_{\Psi^{(K)}(\HOM^{(K)})}=\Hm^{(K)}\circ (\Psi^{(K)})^{-1}$.
\end{definition}

The next theorem, which is the main result of this section, states a powerful characterization of the convergence of $A$-sampling-scale sequences, as well as a description of the resulting measure when applicable.

\begin{thm}
 The $A$-sampling-scale sequence $(S_n)_{n\in\mathbb{N}}$ is convergent if and only if $\mu_Q^{(K)}$ is invariant on infinite $K\in{}^*\mathbb{N},$ in which case $\HMU=\mu_Q^{(K)}$, $\HU=U_Q^{(K)}$ and $\Hm=m_Q^{(K)}$ $\HMU$-almost everywhere.
\end{thm}
\begin{proof}
	First, we note that for any $h\in C(Q)$ and  infinite $K\in{}^*\mathbb{N},$ we have 
	\begin{align*}
		\int_Q hd\mu_Q^{(K)}&=\int_{\HOM^{(K)}}(h\circ \Psi^{(K)})d\HMU^{(K)}=\int_{\Omega_K}(h\circ \Psi^{(K)}\circ \HNU^{(K)})d\mu_L^{(K)}\\
		&=\int_{\Omega_K}(h\circ \St_Q\circ X_K)d\mu_L^{(K)}=\int_{\Omega_K}(\St\circ {}^*h\circ X_K)d\mu_L^{(K)}.
	\end{align*}
	
	Since $Q$ is compact and $h$ is continuous, $|{}^*h(X_K(f))|\leq \max_Q(|h|)\in\mathbb{R}$ whenever $f\in\Omega_K.$ Thus, ${}^*h\circ X_K$ is an internal function on $\Omega_K$ with a finite bound. Theorem 6.1 of \cite{Ross1997} for Loeb measures applies, with $$\int_{\Omega_K}(\St\circ {}^*h\circ X_K)d\mu_L^{(K)}=\St\left(\int_{\Omega_K}({}^*h\circ X_K)d\tilde{\mu}^{(K)}\right).$$ And so we get $$\int_Q hd\mu_Q^{(K)}=\St\left(\int_{\Omega_K}({}^*h\circ X_K)d\tilde{\mu}^{(K)}\right)=\St\left(\int_{\Omega_K}({}^*h\circ X_K)d({}^*\mu_K)\right) $$ for any $h\in C(Q)$ and infinite $K\in{}^*\mathbb{N}.$ Therefore, by elementary nonstandard analysis, we have the characterization of the set of accumulation points $$\operatorname{Acc}\left(\left(\int_{\Omega_n}(h\circ X_n)d\mu_n\right)_{n\in\mathbb{N}}\right)=\left\{\int hd\mu_Q^{(K)}\;|\; K\in{}^*\mathbb{N}\setminus\mathbb{N} \right\}$$ given any $h\in C(Q).$ 
	
	On one hand, if $\mu_Q^{(K)}$ is invariant on $K$, then given any $h\in C(Q),$ the right-hand side is a singleton. Therefore, the left-hand side is a singleton as well, and so the sequence $\left(\int_{\Omega_n}(h\circ X_n)d\mu_n\right)_{n\in\mathbb{N}}$ converges. Thus, $(S_n)_{n\in\mathbb{N}}$ is convergent.
	
	On the other hand, if $(S_n)_{n\in\mathbb{N}}$ is convergent, then the left-hand side is a singleton given any such $h$. In that case, the right-hand side is a singleton as well, and for any infinite $K_1, K_2\in{}^*\mathbb{N},$ $\int_Qhd\mu_Q^{(K_1)}=\int_Qhd\mu_Q^{(K_2)}$. Therefore, reversing our quantifiers, we have that for any infinite $K_1, K_2\in{}^*\mathbb{N},$ $\int_Qhd\mu_Q^{(K_1)}=\int_Qhd\mu_Q^{(K_2)}$ whenever $h\in C(Q).$ By the uniqueness of Riesz-Markov representation, we conclude that $\mu_Q^{(K_1)}=\mu_Q^{(K_2)},$ and so $\mu_Q^{(K)}$ is invariant on $K$. For the same reason, we have that the resulting measure $\HMU$ is given by $\HMU=\mu_Q^{(K)}$ in this case. 
	
	Now, we want to show that in that case $U_Q^{(K)}=\HU$ and $m_Q^{(K)}=\Hm.$ First, we note that if $x\in{}^*H$ is infinitesimal, then $\tilde{U}^{(K)}(x)$ is $\mu_L^{(K)}$-almost everywhere infinitesimal, as (for $x\neq 0$) $\tilde{U}^{(K)}(\frac{x}{\|x\|})$ is almost everywhere finite by Proposition \ref{proposition_nonstandard_limited_isometry}. Thus, for $j\in\mathbb{N}$, we  have that $U_L^{(K)}(e_j)=\St\circ\tilde{U}^{(K)}(e_j^{(K)})$ $\mu_L$-almost everywhere on $\Omega_K$. 
	
	We then evaluate, for any $j\in\mathbb{N}$, and with any $f\in\Omega_K$ such that $$(\HU^{(K)}(e_j))(\HNU^{(K)}(f))=(U_L^{(K)}(e_j))(f)=\St((\tilde{U}^{(K)}(e_j^{(K)}))(f)),$$ which holds $\mu_L$-almost everywhere on $\Omega_K$:
	
	\begin{align*}
		(U_Q^{(K)}(e_j))(\Psi^{(K)}(\HNU^{(K)}(f)))&=(\HU^{(K)}(e_j))(\HNU^{(K)}(f))\\
		&=(U_L^{(K)}(e_j))(f)\\
		&=\St((\tilde{U}^{(K)}(e_j^{(K)}))(f))\\
		&=\St(({}^*U_K(e_j^{(K)}))(f))\\
		&=\frac{1}{\sqrt{c_j}}\St({}^*\pi_j(X_K(f)))\\
		&=\frac{1}{\sqrt{c_j}}\pi_j\left(\St_Q(X_K(f))\right)\\
		&=V_j(\Psi^{(K)}(\HNU^{(K)}(f)))\\
		&=(\HU(e_j))(\Psi^{(K)}(\HNU^{(K)}(f))).
	\end{align*}
	Thus, since $\HMU(\Psi^{(K)}(\HOM^{(K)}))=\HMU^{(K)}(\HOM^{(K)})=1,$ we have $U_Q^{(K)}(e_j)=\HU(e_j).$ Since $\{e_j\}_{j\in\mathbb{N}}$ span a dense subset of $H$, we conclude $U_Q^{(K)}=\HU.$
	
	Finally, we calculate, for any $x\in\Dom(A),$ that  $$\HU(Ax)=\HU^{(K)}(Ax)\circ (\Psi^{(K)})^{-1}= (\Hm^{(K)}\cdot \HU^{(K)}(x))\circ (\Psi^{(K)})^{-1}=m_Q^{(K)}\cdot \HU(x). $$ Since this property characterizes $\Hm$ by Theorem \ref{theorem_multiplication_operator}, we conclude $m_Q^{(K)}=\Hm$.
\end{proof}

One reason this characterization seems powerful is given by the following.
\begin{cor}\label{corollary_nonstandard_invariance}
	Suppose that $(\Omega, \mathcal{A}, \mu)$ is some probability space. Furthermore, suppose that for any infinite $K\in{}^*\mathbb{N}$, we have a measurable map $\hat{Y}^{(K)}:\Omega\rightarrow\HOM^{(K)}$ for which $\hat{Y}^{(K)}_{\#}(\mu)=\HMU^{(K)}.$ Finally, suppose that $\Psi^{(K)}\circ \hat{Y}^{(K)}=X:\Omega\rightarrow Q$ is invariant on $K$. 
	
	Then, $(S_n)_{n\in\mathbb{N}}$ is convergent, with $\HMU=X_{\#}(\mu)$. Furthermore, we have that for any $x\in\Dom(A),$  $$U(Ax)=m\cdot U(x),$$ where $m=\Hm\circ X=\Hm^{(K)}\circ \hat{Y}^{(K)}: \Omega\rightarrow\mathbb{R}$, and $U: H\rightarrow L_2(\Omega,\mu)$ is the isometry given by $U(z)=\HU(z)\circ X=\HU^{(K)}(z)\circ \hat{Y}^{(K)}$ for $z\in H$. 
\end{cor}
\begin{proof}
	Let $(\Omega,\mathcal{A},\mu)$ be a probability space, for which we have the maps $\hat{Y}^{(K)}$ satisfying the suppositions, so that $X=\Psi^{(K)}\circ \hat{Y}^{(K)}$ is well defined.
	
	Then, we note that for any infinite $K\in{}^*\mathbb{N},$ we have that $$\mu_Q^{(K)}=\Psi^{(K)}_{\#}(\HMU^{(K)})=\Psi^{(K)}_{\#}(\hat{Y}^{(K)}_{\#}(\mu))=X_{\#}(\mu).$$ The rightmost term does not depend on $K$ by hypothesis, and so $\mu_Q^{(K)}$ is invariant on $K$. Therefore, by the last theorem, $(S_n)_{n\in\mathbb{N}}$ is convergent with $\HMU=X_{\#}(\mu),$ $\HU=U_Q^{(K)}$ and $\Hm=m_Q^{(K)}.$
	
	Furthermore, let $m=\Hm\circ X$ and $U(z)=\HU(z)\circ X$  for $z\in H$, noting that $U$ is an isometry since $X$ is measure-invariant. Using Theorem \ref{theorem_multiplication_operator}, we have that for any $x\in\Dom(A),$
	\begin{align*}
		U(Ax)=(\HU(Ax))\circ X=(\Hm\cdot \HU(x))\circ X=(\Hm\circ X)\cdot ((\HU(x))\circ X)=m\cdot U(x).
	\end{align*} 
	Finally, we have that $$m=\Hm\circ X=m_Q^{(K)}\circ \Psi^{(K)}\circ \hat{Y}^{(K)}=\Hm^{(K)}\circ\hat{Y}^{(K)}$$, and for any $z\in H$, $$U(z)=\HU(z)\circ X=U_Q^{(K)}(z)\circ\Psi^{(K)}\circ \hat{Y}^{(K)}=\HU^{(K)}(z)\circ \hat{Y}^{(K)},$$ concluding the proof.
\end{proof}
\begin{remark}
	That last corollary will prove to be surprisingly useful, as we will use it to quickly prove the main results of both example, allowing us to bypass much of the technical details. We do note that to make use of these results, one must do the work in the setting of \cite{nonez2024spectralequivalencesnonstandardsamplings} first.
\end{remark}
\section{Proof of Theorem \ref{theorem_shift_fourier_series} }\label{appendix_shift}

We now use the same terminology for $H$, $A$ and the $A$-sampling-scale sequence $(S_n)_{n\in\mathbb{N}}$ described in \ref{section_shift}, in order to prove Theorem \ref{theorem_shift_fourier_series}. As such, we assume $\mu$ is the Lebesgue measure on $\Borel([0,1]),$ and $X:[0,1]\rightarrow Q$ is such that $$X(t)=\left(\frac{1}{\sqrt{2^j}}e^{2\pi i l_j t}\right)_{n\in\mathbb{N}}.$$

In Section 6 of \cite{nonez2024spectralequivalencesnonstandardsamplings}, the same example is considered, precisely with $H=\ell_2(\mathbb{Z})$ on $\mathbb{K}=\mathbb{C}$ and $A=\frac{1}{2}(LS+RS)$. The chosen sampling and scale is given, for infinite $K\in{}^*\mathbb{N}$, by:
\begin{itemize}
	\item $\tilde{H}={}^*\Span(\{{}^*e_j \}_{j=1}^{2K+1})=H_K$;
	\item $\tilde{A}=\frac{1}{2}({}^*RS_K+{}^*LS_K)=A_K$;
	\item $\tilde{\Omega}=\{{}^*f_k^{(K)} \}_{k=1}^{2K+1}=\{{}^*f_k^{(K)} \}_{k=0}^{2K}=\Omega_K$;
	\item $\tilde{N}=2K+1=N_K$;
	\item $\tilde{e}_j= {}^*e_j= e_j^{(K)}$, for $j\leq N_K$;
	\item $\tilde{c}_j=\frac{1}{2^j(1-2^{-N_K})}=c_j^{(K)},$
\end{itemize}
and thus the sampling is $(H_K,A_K,\Omega_K)$ and the scale is $(e_j^{(K)},c_j^{(K)})_{j=1}^{N_K},$ for any arbitrary infinite $K$. We now state the key results of the calculations done in \cite{nonez2024spectralequivalencesnonstandardsamplings}.

\begin{prop}
	We have that for any $l\in\mathbb{Z}$ and $f_k^{(K)}\in \Omega_K,$ $$(\tilde{U}^{(K)}(g_l))(f_k^{(K)})=e^{2\pi i l\frac{k}{N_K}},$$ and so for any $l\in\mathbb{Z}$ we can formulate $$(\HU^{(K)}(g_l))(\HNU^{(K)}(f_k^{(K)}))=e^{2\pi i l \St(\frac{k}{N_K})}$$ and $$\Hm^{(K)}(\HNU^{(K)}(f_k^{(K)}))=\cos(2\pi\St(\frac{k}{N_K})).$$
	Furthermore, if $\mathbb{R}/\mathbb{Z}$ is endowed with the Lebesgue measure on its Borel sets, then the map $\rho^{(K)}:\HOM^{(K)}\rightarrow \mathbb{R}/\mathbb{Z}$ given by $\rho^{(K)}(\HNU^{(K)}(f_k^{(K)}))=\St(\frac{k}{N_K}) \operatorname{mod}\;1$ is a well-defined measure preserving homeomorphism.
\end{prop}

We can now begin the main proof of this section.

\begin{proof}[Proof of Theorem \ref{theorem_shift_fourier_series}]
	We note that $t\rightarrow t\operatorname{mod}\;1: [0,1]\rightarrow \mathbb{R}/\mathbb{Z}$ is also measure preserving. Thus, we define, for any given infinite $K\in{}^*\mathbb{N}$, $\hat{Y}^{(K)}:[0,1]\rightarrow \HOM^{(K)}$ with $\hat{Y}^{(K)}(t)=(\rho^{(K)})^{(-1)}(t\operatorname{mod}\;1),$  and so $\hat{Y}^{(K)}$ is measure preserving and we have $\HMU^{(K)}=\hat{Y}^{(K)}_{\#}(\mu)$.
	
	We now show that $\Psi^{(K)}\circ \hat{Y}^{(K)}=X$ in order to apply Corollary \ref{corollary_nonstandard_invariance}. Let $t\in[0,1]$ be arbitrary.  Furthermore, let $0\leq k\leq 2K$ being any integer such that $\St(\frac{k}{2K+1})=t$. We note that $\rho^{(K)}(\HNU^{(K)}(f_k^{(K)}))=t\operatorname{mod}\;1.$ Therefore,
	\begin{align*}
		\Psi^{(K)}(\hat{Y}^{(K)}(t))&=\Psi^{(K)}(\HNU^{(K)}(f_k^{(K)}))=\St_Q(X_K(f_k^{(K)})),
	\end{align*}
	and so for any $j\in\mathbb{N}$, and since $\pi_j$ is continuous on $Q$ and $j\leq N_K,$
	\begin{align*}
		\pi_j(\Psi^{(K)}(\hat{Y}^{(K)}(t)))&=\St({}^*\pi_j(X_K(f_k^{(K)})))\\
		&=\St\left(\sqrt{c_j^{(K)}}\left({}^*U_K(e_j^{(K)})\right)(f_k^{(K)})\right)\\	&=\St\left(\sqrt{c_j^{(K)}}\left(\tilde{U}^{(K)}(g_{l_j})\right)(f_k^{(K)})\right)\\
		&=\St\left(\frac{1}{\sqrt{2^j(1-2^{-N_K})}}e^{2\pi il_j \frac{k}{N_K}}\right)=\frac{1}{\sqrt{2^j}}e^{2\pi i l_j t}=\pi_j(X(t)).
	\end{align*}
	Since $j\in\mathbb{N}$ is arbitrary, we have that $\Psi^{(K)}(\hat{Y}^{(K)}(t))=X(t)$, and since $t\in[0,1]$ is arbitrary, we have $\Psi^{(K)}\circ \hat{Y}^{(K)}=X$. In particular, $\Psi^{(K)}\circ \hat{Y}^{(K)}$ is invariant on $K$, and Corollary \ref{corollary_nonstandard_invariance} applies.
	
	In other words, $(S_n)_{n\in\mathbb{N}}$ is convergent, and $\HMU=X_{\#}(\mu)$. Furthermore, for any $z\in H$, $U(Az)=m\cdot U(z)$ up to null Lebesgue measure, where $m=\Hm\circ X$  and $U:H\rightarrow L_2([0,1],\mu)$ is the isometry given by $U(z)=\HU(z)\circ X$ for $z\in H$. 
	
	To conclude the proof, we only need to show that for almost all $t\in[0,1],$ $m(t)=\cos(2\pi t)$ and $(U(g_l))(t)=e^{2\pi il t}$  any $l\in\mathbb{Z}$.
	
	Let $K$ be some infinite hypernatural. Thanks to Corollary \ref{corollary_nonstandard_invariance}, we know that for any $l\in\mathbb{N},$ $(U(g_l)(t))=(\HU^{(K)}(g_l)\circ \hat{Y}^{(K)})(t)$ for almost all $t\in[0,1]$. For any such $t$ where this holds, and fixing hyperinteger $k$ such that $0\leq k\leq 2K$ and $t=\St(\frac{k}{N_K}),$ we calculate:
	
	\begin{align*}
		(U(g_l)(t))&=(\HU^{(K)}(g_l)\circ \hat{Y}^{(K)})(t)=(\HU^{(K)}(g_l))((\rho^{(K)})^{(-1)}(t\operatorname{mod} 1))\\
		&=(\HU^{(K)}(g_l))(\HNU^{(K)}(f_k^{(K)}))=e^{2\pi i l \St(\frac{k}{N_K})}=e^{2\pi i l t}.
	\end{align*}
	
	The same way, we have that $m(t)=(\Hm^{(K)}\circ \hat{Y}^{(K)})(t)$ holds for almost all $t\in[0,1]$ as a consequence of Corollary \ref{corollary_nonstandard_invariance}. Thus, for any such $t$ where it holds and suitable hyperinteger $k$ for which $0\leq k\leq 2K$ and $t=\St(\frac{k}{N_K}),$ we calculate:
	\begin{align*}
		m(t)&=(\Hm^{(K)}\circ \hat{Y}^{(K)})(t)=\Hm^{(K)}((\rho^{(K)})^{(-1)}(t\operatorname{mod} 1))\\
		&=\Hm^{(K)}(\HNU^{(K)}(f_k^{(K)}))=\cos(2\pi\St(\frac{k}{N_K}))=\cos(2\pi t),
	\end{align*}
	concluding the proof.
\end{proof}

\section{Proof of Theorem \ref{theorem_differential_fourier_transform}}\label{appendix_differential}

We now use the same terminology as Section \ref{section_differential}, with the scope of this section being the proof of Theorem \ref{theorem_differential_fourier_transform}. Specifically, we consider $\mathbb{K}=\mathbb{C},$ $H=L_2(\mathbb{R},\mu)$, $A=-i\frac{d}{dx}$ on $\Dom(A)=C_c^{{\infty}}(\mathbb{R})$ and the $A$-sampling-scale sequence $(S_n)_{n\in\mathbb{N}}$ described in Subsections  \ref{subsection_differential_sampling} and \ref{subsection_differential_scale}. We also consider $\mu'$ as the measure on $\Borel(\mathbb{R})$ given by $d\mu'=\varphi d\mu=\left(\frac{\pi}{2}\right)^{\frac{1}{2}}e^{-\frac{\pi^2\omega^2}{2}}d\mu,$ and the map $X:\mathbb{R}\rightarrow Q$ given by $$X(\omega)=\left(\frac{1}{\sqrt{2}^j}e^{-\pi i q_j \omega}\right)_{j\in\mathbb{N}},$$ $(q_j)_{j\in\mathbb{N}}$ being a count for $\mathbb{Q}$ with $q_1=0.$

The same example $H$ and $A$ are considered in Section 7 of \cite{nonez2024spectralequivalencesnonstandardsamplings}. The sampling is then given, for any infinite $K\in{}^*\mathbb{N}:$
\begin{itemize}
	\item $\tilde{H}={}^*\Span(\{\mathbf{1}_{s_l^{(K)}}\}_{l=-K!^2}^{K!^2-1})=H_K$;
	\item $\tilde{A}=-i\frac{\tilde{L}-\tilde{R}}{2/K!}=-i\frac{{}^*LS_K-{}^*RS_K}{2/K!}=A_K$;
	\item $\tilde{\Omega}=\{f_k^{(K)}\}_{k=-K!^2}^{K!^2-1}=\Omega_K$;
\end{itemize}

For the scale, there are a few preliminary definitions first, which mostly coincide with the ones of Subsection \ref{subsection_differential_scale}. First, $\tilde{N}_1$ is established as any infinite hypernatural for which $\sqrt{K!}<\frac{\tilde{N}_1}{K!}+q<K!$ for any standard $q\in\mathbb{Q}.$ We have that $L_K$ described in Subsection \ref{subsection_differential_scale} satisfy this criterion, so we use $\tilde{N}_1=L_K$. From there, $\tilde{e}$ is defined as $$\tilde{e}=\left(\frac{2}{\pi}\right)^{\frac{1}{4}}\sum_{k=-\tilde{N}_1}^{\tilde{N}_1}e^{(-(\frac{k}{K!})^2)}\mathbf{1}_{s_k}+\frac{i}{K!}\mathbf{1}_{s_0},$$ and so we find $\tilde{e}=E^{(K)}.$ Furthermore, $\{k_j\}_{j=1}^{K!}$ is defined as any internal sequence in ${}^*\mathbb{Z}$ such that for any $j\in\mathbb{N},$ $k_j=K! q_j.$ We use $k_j=\lfloor K! q_j \rfloor$ for any $j\leq K!.$ With that, we have:

\begin{itemize}
	\item $\tilde{N}=K!=N_K$;
	\item $\tilde{e}_j=\tilde{R}^{k_j}\tilde{e}=RS_K^{\lfloor K! q_j \rfloor}E^{(K)}=e_j^{(K)}$;
	\item $\tilde{c}_j=\frac{1}{2^j\|\tilde{e}\|^2(1-2^{-K!})}=c_j^{(K)}.$
\end{itemize}

Thus, for any arbitrary infinite $K$, the sampling is $(H_K, A_K. \Omega_K)$ and the scale is $(e_j^{(K)},c_j^{(K)})_{j=1}^{N_K}$. We now summarize key results of the calculations done in \cite{nonez2024spectralequivalencesnonstandardsamplings}, in terms of the notation presented here.

\begin{prop}\label{proposition_nsa_differential}
	Let $\tilde{\Omega}_{\mathbb{R}}^{(K)}=\{f_k^{(K)}\in\Omega_K\;|\; \St(\frac{k}{N_K})\in\mathbb{R}\},$ and let $\hat{\Omega}_{\mathbb{R}}^{(K)}=\HNU^{(K)}(\tilde{\Omega}_{\mathbb{R}}^{(K)}).$ We have that $\mu_L^{(K)}(\HOM^{(K)}\setminus\hat{\Omega}_{\mathbb{R}}^{(K)})=0.$ Furthermore, we have that for any $j\in\mathbb{N}$ and $f_k^{(K)}\in\Omega_K,$ $$(\tilde{U}^{(K)}(e_j^{(K)}))(f_k^{(K)})=e^{-\pi i q_j\frac{k}{N_K}}.$$
	
	Then, the map $\rho^{(K)}:\mathbb{R}\rightarrow \hat{\Omega}_{\mathbb{R}}^{(K)}$ given by $\rho^{(K)}(\St(\frac{k}{N_K}))=\HNU^{(K)}(f_k^{(K)})$ is well-defined on $\mathbb{R}$, as well as continuous and bijective. Furthermore, we have that $(\rho^{(K)})^{-1}$ is also measurable, and $\left((\rho^{(K)})^{-1}\right)_{\#}(\HMU^{(K)})=\mu'.$ We also have that $\HU^{(K)}$ is unitary.
	
	Finally, we can formulate $$(\HU^{(K)}(\psi))(\rho^{(K)}(\omega))=\frac{1}{\sqrt{2\varphi(\omega)}}\int_{\mathbb{R}}e^{-\pi i t\omega }\psi(t)d\mu(t)$$ for any $\psi\in L_1(\mathbb{R},\mu)\cap L_2(\mathbb{R},\mu) $ and $\mu'$-almost all $\omega\in \mathbb{R}$, as well as	$$\Hm^{(K)}(\rho^{(K)}(\omega))=\pi\omega$$ for $\mu'$-almost all $\omega\in\mathbb{R} .$
\end{prop}

We can now complete the proof of this section.

\begin{proof}[Proof of  Theorem \ref{theorem_differential_fourier_transform}]
	We want to apply  Corollary \ref{corollary_nonstandard_invariance} with the measure space $(\mathbb{R},\Borel(\mathbb{R},\mu')$ and $\hat{Y}^{(K)}:\mathbb{R}\rightarrow\Omega_K$ given by $\hat{Y}^{(K)}(t)=\rho^{(K)}(t)$. Since $\rho^{(K)}$ is a measureable bijection with measure preserving inverse and full-measure image, we have that $\rho^{(K)}$ is also measure-preserving, so that $\hat{Y}^{(K)}_{\#}(\mu')=\HMU^{(K)}.$ To apply the corollary, it is sufficient to show that $\Psi^{(K)}\circ \hat{Y}^{(K)}=X.$
	
	But then, given $\omega\in\mathbb{R},$ for any $f_k^{(K)}\in\tilde{\Omega}_{\mathbb{R}}^{(K)}$ such that $\St(\frac{k}{N_K})=\omega$, and $j\in\mathbb{N}$, we have
	\begin{align*}
		\pi_j(\Psi^{(K)}(\hat{Y}^{(K)}(\omega)))&=\pi_j(\Psi^{(K)}(\HNU^{(K)}(f_k^{(K)})))\\
		&=\pi_j(\St_Q(X_K(f_k^{(K)})))\\
		&=\St({}^*\pi_j(X_K(f_k^{(K)})))\\
		&=\St(\sqrt{c_j^{(K)}}\left({}^*U_K(e_j^{(K)})\right)(f_k^{(K)}))\\
		&=\St\left(\sqrt{\frac{1}{2^j\|E^{(K)}\|^2(1-2^{-K!})}}\left(\tilde{U}^{(K)}(e_j^{(K)})\right)(f_k^{(K)})\right)\\
		&=\frac{1}{\sqrt{2^j}}\St(e^{-\pi i q_j\frac{k}{N_K}})\\
		&=\frac{1}{\sqrt{2^j}}e^{-\pi i q_j\omega}\\
		&=\pi_j(X(\omega)).
	\end{align*}
	Since both $j\in\mathbb{N}$ and $\omega\in\mathbb{R}$ are arbitrary, we conclude that $\Psi^{(K)}\circ \hat{Y}^{(K)}=X,$ and so Corollary \ref{corollary_nonstandard_invariance} applies.
	
	Thus, $(S_n)_{n\in\mathbb{N}}$ is convergent, with $\HMU=X_{\#}(\mu')$. Furthermore, for any $z\in H$, $U(Az)=m\cdot U(z)$ in $L_2(\mathbb{R},\mu')$, where $m=\Hm\circ X$  and $U:H\rightarrow L_2(\mathbb{R},\mu')$ is the isometry given by $U(z)=\HU(z)\circ X$ for $z\in H$. 
	
	We note that since $\HU^{(K)}$ is unitary, so is $U$. Indeed, with Corollary \ref{corollary_nonstandard_invariance}, for any $g\in L_2(\mathbb{R},\mu')$, we have $g=U((\HU^{(K)})^{-1}(g\circ (\rho^{(K)})^{-1}))$, making $U$ surjective. 
	
	With Proposition \ref{proposition_nsa_differential}, we  have, for almost all $\omega\in\mathbb{R}$, that $$m(\omega)=\Hm^{(K)}(\hat{Y}^{(K)}(\omega))=\pi\omega.$$ Finally, we also have, for any $\psi\in L_1(\mathbb{R},\mu)\cap L_2(\mathbb{R},\mu)$ and almost all $\omega\in\mathbb{R}$, that $$(U(\psi))(\omega)=(\HU^{(K)}(\psi))(\hat{Y}^{(K)}(\omega))=\frac{1}{\sqrt{2\varphi(\omega)}}\int_{\mathbb{R}}e^{-\pi i t\omega }\psi(t)d\mu(t),$$ concluding the proof.
\end{proof}
\section*{Acknowledgments}
I extend my sincere gratitude to Isaac Goldbring for the helpful discussions and advice for this paper.
\printbibliography
\end{document}